\documentclass[a4paper,11pt,twoside]{article}
\usepackage[all]{xy}
\usepackage{graphicx}
\usepackage{pgf,tikz}
\usepackage{mathrsfs}
\usetikzlibrary{arrows}
\usepackage{lipsum}
\usepackage{mathtools}

\DeclarePairedDelimiter\floor{\lfloor}{\rfloor}

\definecolor{black}{cmyk}{1.,1.,1.,1.0}
\definecolor{blue}{cmyk}{1.,1.,0.,0.63}
\definecolor{green}{cmyk}{1.,0.,1.,0.63}

\newcommand{\red}{\textcolor{red}}

\usepackage{hyperref}
\hypersetup{
       unicode=true,          
     pdftoolbar=true,        
    pdfmenubar=true,        
    pdffitwindow=false,     
    pdfstartview={FitH},    
       pdfproducer={}, 
    pdfcreator={pdfLaTeX},   
    pdfproducer={}, 
    pdfnewwindow=true,      
    colorlinks=false,       
    linkcolor=blue,          
    citecolor=green,        
    filecolor=magenta,      
    urlcolor=cyan           
}
\usepackage[left=2cm,right=2cm,top=2cm,bottom=1.8cm]{geometry}
\usepackage{amssymb,amsmath,amsthm}

\DeclareMathOperator{\der}{d}

\DeclareMathOperator{\ord}{ord}

\newcommand{\vf}{\vfill\end{document}}
\newcommand{\explain}[1]{\text{\scriptsize\sf [#1]}}

\newtheorem{thm}{Theorem}[section]
\newtheorem{lem}{Lemma}[section]

\newtheorem{claim}{Claim}[section]
\theoremstyle{plain}
\numberwithin{equation}{section}
\newcommand{\thistheoremname}{}
\newtheorem*{genericthm*}{\thistheoremname}
\newenvironment{namedthm*}[1]{\renewcommand{\thistheoremname}{#1}%
\begin{genericthm*}}
{\end{genericthm*}}

\newtheoremstyle{named}{}{}{\itshape}{}{\bfseries}{.}{.5em}{\thmnote{#3's }#1}
\theoremstyle{named}

\title{EXAMPLES OF HYPERBOLIC HYPERSURFACES
\\
OF LOW DEGREE
\\
IN PROJECTIVE SPACES}
\providecommand{\keywords}[1]{\textbf{\textit{Keywords:}} #1}

\author{Dinh Tuan HUYNH}

\newcommand{\Addresses}{{
\bigskip
\footnotesize
\textsc{Dinh Tuan Huynh, Laboratoire de Math\'{e}matiques d'Orsay, Univ. Paris-Sud, CNRS, Universit\'{e} Paris-Saclay, 91405 Orsay, France.}\\
\textsc{Department of Mathematics, College of Education, Hue University, 34 Le Loi St., Hue City, Vietnam.}
\par\nopagebreak
\textit{E-mail address}: \texttt{dinh-tuan.huynh@math.u-psud.fr}}}
\date{\vspace{-5ex}}
\begin{document}
\maketitle
\begin{abstract}
We construct families of hyperbolic hypersurfaces of degree $2n$ in the projective space $\mathbb{P}^n(\mathbb{C})$ for $3\leq n\leq 6$.
\end{abstract}
, {hyperbolicity}, {Brody Lemma}, {Nevanlinna Theory}

\section{Introduction and the main result}

The Kobayashi conjecture states that a generic hypersurface $X_d\subset\mathbb{P}^n(\mathbb{C})$ of degree $d\geq 2n-1$ is hyperbolic. It is proved by Demailly and El Goul \cite{demailly_goul2000} for $n=3$ and a very generic surface of degree at least $21$.  In \cite{mihaipaun2008}, P\u{a}un improved the degree to $18$. In $\mathbb{P}^4(\mathbb{C})$, Rousseau \cite{Rousseau2007} was able to show that a generic three-fold of degree at least $593$ contains no Zariski-dense entire curve, a result from which hyperbolicity follows, after removing divisorial components \cite{Diverio-Trapani2010}. In $\mathbb{P}^n(\mathbb{C})$, for any $n$ and for $d \geq2^{(n-1)^5}$, Diverio, Merker and Rousseau \cite{DMR2010} established algebraic degeneracy of entire curves in $X_d$. Also, Siu \cite{siu2015} proposed a positive answer for arbitrary $n$ and degree $d=d(n)\gg 1$ very large. Most recently, Demailly \cite{demailly2015proof} has announced a strategy that is expected to attain Kobayashi's conjecture for {\em very} generic hypersurfaces of degree $d\geq2n$.

Concurrently, many authors tried to find examples of hyperbolic hypersurfaces of degree as low as possible. The first example of a {\em compact} Kobayashi hyperbolic manifold of dimension $2$ is a hypersurface in $\mathbb{P}^3(\mathbb{C})$ constructed by Brody and Green \cite{Brody-Green1977}. Also, the first examples in all higher dimensions $n-1\geq 3$ were discovered by Masuda and Noguchi \cite{masuda_noguchi1996}, with degree large. So far, the best degree asymptotic is the square of dimension, given by Siu and Yeung \cite{siu_yeung1997} with $d=16(n-1)^2$ and by Shiffman and Zaidenberg \cite{shiffman_zaidenberg2002_pn} with $d=4(n-1)^2$. In $\mathbb{P}^3(\mathbb{C})$ many examples of low degree were given (see the reference of \cite{zaidenberg2003hyperbolic}). The lowest degree found up to date is $6$, given by Duval \cite{duval2004}. There are not so many examples of low degree hyperbolic hypersurfaces in $\mathbb{P}^4(\mathbb{C})$. We mention here an example of a hypersurface of degree $16$ constructed by Fujimoto \cite{fujimoto2001}. Various examples in $\mathbb{P}^5(\mathbb{C})$ and $\mathbb{P}^6(\mathbb{C})$ only appear in the cases of arbitrary dimension mentioned above.

Before going to introduce the main result, we need some notations and conventions. A family of hyperplanes $\{H_i\}_{1\leq i\leq q}$ with $q\geq n+1$ in $\mathbb{P}^n(\mathbb{C})$ is said to be in \textsl{general position} if any $n+1$ hyperplanes in this family have empty intersection. A hypersurface $S$ in $\mathbb{P}^n(\mathbb{C})$ is said to be in \textsl{general position with respect to $\{H_i\}_{1\leq i\leq q}$} if it avoids all intersection points of $n$ hyperplanes in $\{H_i\}_{1\leq i\leq q}$, namely if:
\[
S\cap\big(\cap_{i\in I}H_i\big)=\emptyset,\quad\forall I\subset \{1,\ldots, q\}, |I|=n.
\]

Now assume that $\{H_i\}_{1\leq i\leq q}$ is a family of hyperplanes of $\mathbb{P}^n(\mathbb{C})$ ($n\geq 2$) in general position. Let $\{H_i\}_{i\in I}$ be a subfamily of $n+2$ hyperplanes. Take a partition $I=J\cup K$ such that $|J|,|K|\geq 2$. Then there exists a unique hyperplane $H_{JK}$ containing $\cap_{j\in J}H_j$ and $\cap_{k\in K}H_k$. We call $H_{JK}$ a \textsl{diagonal hyperplane} of $\{H_i\}_{i\in I}$. The family $\{H_i\}_{1\leq i\leq q}$ is said to be \textsl{generic} if, for all disjoint subsets $I,J,J_1,\ldots,J_k$ of $\{1,\ldots,q\}$ such that $|I|, |J_i|\geq 2$ and $|I|+|J_i|=n+2$, $1\leq i\leq k$, for every subset $\{i_1,\ldots,i_l\}$ of $I$, the intersection between the $|J|$ hyperplanes $H_j,j\in J$, the $k$ diagonal hyperplanes $H_{IJ_1},\ldots,H_{IJ_k}$, and the $l$ hyperplanes $H_{i_1},\ldots,H_{i_l}$ is a linear subspace of codimension $\min\{k+l,|I|\}+|J|$, with the convention that when $\min\{k+l,|I|\}+|J|>n$, this intersection is empty. Such a generic condition naturally appears in our constructions,  and it has the virtue of being preserved when passing to smaller-dimensional subspaces\\

 Our aim in this article is to prove that, for $3\leq n\leq 6$, a small deformation of a union of generic $2n$ hyperplanes in $\mathbb{P}^n(\mathbb{C})$ is hyperbolic.
 
\begin{namedthm*}{Main Theorem}\label{maintheorem}
Let $n$ be an integer number in $\{3,4,5,6\}$. Let $\{H_i\}_{1\leq i\leq 2n}$  be a family of $2n$ generic hyperplanes in $\mathbb{P}^n(\mathbb{C})$, where $H_i=\{h_i=0\}$. Then there exists a hypersurface $S=\{s=0\}$ of degree $2n$ in general position with respect to $\{H_i\}_{1\leq i\leq 2n}$ such that the hypersurface 
\[
\Sigma_{\epsilon}
=
\big\{\epsilon s+\Pi_{i=1}^{2n}h_i=0\big\}
\]
is hyperbolic for sufficiently small complex $\epsilon\not=0$.
\end{namedthm*}

Our proof is based on the technique of Duval \cite{duval2014} in the case $n=3$. By the deformation method of Zaidenberg and Shiffman\cite{shiffman_zaidenberg2005}, the problem reduces to finding a hypersurface $S$ such that all complements of the form
\[ 
\cap_{i\in I}H_i\setminus\big(\cup_{j\not\in I}H_j\setminus S\big)
\]
are hyperbolic. This situation is very close to Theorem~\ref{fujimoto-green}. To create such $S$, we proceed by deformation in order to allow points of intersection of $S$ with more and more linear subspaces coming from the family $\{H_i\}_{1\leq i\leq 2n}$.

\section*{Acknowledgments}
This is a part of my Ph.D. thesis at Universit\'{e} Paris-Sud. I would like to gratefully thank my thesis advisor Julien Duval for his support, his very careful readings and many inspiring discussions on the subject. I am specially thankful to my thesis co-advisor Jo\"{e}l Merker for his encouragements, his prompt helps in {\LaTeX} and his comments that greatly improved the manuscript. I would like to thank Junjiro Noguchi for his bibliographical help. Finally, I acknowledge support from Hue University - College of Education.

\section{Notations and preparation}
\subsection{Brody Lemma}
Let $X$ be a compact complex manifold equipped with a hermitian metric $\| \cdot\|$. By an \textsl{entire curve} in $X$ we mean a nonconstant holomorphic map $f:\mathbb{C}\rightarrow X$. A \textsl{Brody curve} in $X$ is an entire curve $f:\mathbb{C}\rightarrow X$ such that $\| f'\|$ is bounded. Brody curves arise as limits of sequences of holomorphic maps as follows (see \cite{brody1978}).

\begin{namedthm*}{Brody Lemma}
\label{brodylemma}
Let $f_n:\mathbb{D}\rightarrow X$ be a sequence of holomorphic maps from the unit disk to a compact complex manifold $X$. If $\|f_n'(0)\|\rightarrow\infty$ as $n\rightarrow \infty$, then there exist a point $a\in\mathbb{D}$, a sequence $(a_n)$ converging to $a$, and a decreasing sequence $(r_n)$ of positive real numbers converging to 0 such that the sequence of maps
\[
z\rightarrow f_n(a_n+ r_n z)
\]
converges toward a Brody curve, after extracting a subsequence.
\end{namedthm*}

From an entire curve in $X$, the Brody Lemma also produces a Brody curve in $X$. A second consequence is a well-known characterization of Kobayashi hyperbolicity.

\begin{namedthm*}{Brody Criterion}\label{brody_criterion}
A compact complex manifold $X$ is Kobayashi hyperbolic if and only if it contains no entire curve (or no Brody curve).
\end{namedthm*}

We shall repeatedly use the Brody Lemma under the following form.

\begin{namedthm*}{Sequences of entire curves}\label{lemma brody_use}

Let $X$ be a compact complex manifold and let $(f_n)$ be a sequence of entire curves in $X$. Then there exist a sequence of reparameterizations $r_n:\mathbb{C}\rightarrow\mathbb{C}$ and a subsequence of $(f_n\circ r_n)$ which converges towards an entire curve (or Brody curve).
\end{namedthm*}

\subsection{Nevanlinna theory and some applications}
We recall some facts from Nevanlinna theory in the projective space $\mathbb{P}^n(\mathbb{C})$. Let $E=\sum\mu_{\nu}a_{\nu}$ be a divisor on $\mathbb{C}$ and let
$k\in \mathbb{N}\cup\{\infty\}$. Summing the $k$-truncated degrees of the divisor on disks by
\[
n^{[k]}(t,E)
:=
\sum_{|a_{\nu}|<t}
\min
\,
\{k,\mu_{\nu}\}\quad\quad 
{\scriptstyle (t\,>\,0)},
\]
the \textsl{truncated counting function at level} $k$ of $E$ is defined by
\[
N^{[k]}(r,E)
\,
:=
\,
\int_1^r \frac{n^{[k]}(t, E)}{t}\,\der t\quad \quad
{\scriptstyle (r\,>\,1)}.
\]
When $k=\infty$, we write $n(t,E)$, $N(r,E)$ instead of $n^{[\infty]}(t,E)$, $N^{[\infty]}(r,E)$. We denote the zero divisor of a nonzero meromorphic function $\varphi$ by $(\varphi)_{0}$. Let $f\colon\mathbb{C}\rightarrow \mathbb{P}^n(\mathbb{C})$ be an entire curve having a reduced representation $f=[f_0:\cdots:f_n]$ in the homogeneous coordinates $[z_0:\cdots:z_n]$ of $\mathbb{P}^n(\mathbb{C})$. Let $D=\{Q=0\}$ be a hypersurface in $\mathbb{P}^n(\mathbb{C})$ defined by a homogeneous polynomial $Q\in\mathbb{C}[z_0,\dots,z_n]$ of degree $d\geq 1$. If $f(\mathbb{C})\not\subset D$, we define the \textsl{truncated counting function} of $f$ with respect to $D$ as
\[
N_f^{[k]}(r,D)
\,
:=
\,
N^{[k]}\big(r,(Q\circ f)_0\big).
\]
The \textsl{proximity function} of $f$ for the divisor $D$ is defined as
\[
m_f(r,D)
\,
:=
\,
\int_0^{2\pi}
\log
\frac{\big\Vert f(re^{i\theta})\big\Vert^d\,
\Vert Q\Vert}{\big|Q(f)(re^{i\theta})\big|}
\,
\frac{\der\theta}{2\pi},
\]
where $\Vert Q\Vert$ is the maximum  absolute value of the coefficients of $Q$ and
\[
\big\Vert f(z)\big\Vert
\,
=
\,
\max
\,
\{|f_0(z)|,\ldots,|f_n(z)|\}.
\]
Since $\big|Q(f)\big|\leq \Vert Q\Vert\cdot\Vert f\Vert^d$, one has $m_f(r,D)\geq 0$. Finally, the \textsl{Cartan order function} of $f$ is defined by
\[
T_f(r)
\,
:=
\,
\frac{1}{2\pi}\int_0^{2\pi}
\log
\big\Vert f(re^{i\theta})\big\Vert \der\theta.
\]

It is known that \cite{eremenko2010} if $f$ is a Brody curve, then its order
\[
\rho_f
\,
:=
\,
\limsup_{r\rightarrow +\infty}\dfrac{T_f(r)}{\log r}
\]
 is bounded from above by 2. Furthermore, Eremenko \cite{eremenko2010} showed the following.
 
\begin{thm}
\label{omit n hyperplanes}
If $f\colon\mathbb{C}\rightarrow \mathbb{P}^n(\mathbb{C})$ is a Brody curve omitting $n$ hyperplanes in general position, then it is of order 1.
\end{thm}

Consequently, we have the following theorem.
 
\begin{thm}
\label{omit-hyperplanes}
If $f\colon\mathbb{C}\rightarrow\mathbb{P}^n(\mathbb{C})$ is a Brody curve avoiding the first $n$ coordinate hyperplanes $\{z_i=0\}_{i=0}^{n-1}$, then it has a reduced representation of the form
\[
[1:e^{\lambda_1\,z+\mu_1}:\cdots:e^{\lambda_{n-1}\, z+\mu_{n-1}}:g],
\]
where $g$ is an entire function and $\lambda_i$, $\mu_i$ are constants. If $f$ also avoids the remaining coordinate hyperplane $\{z_n=0\}$, then $g$ is of the form $e^{\lambda_n\,z+\mu_n}$.
\end{thm}

The core of Nevanlinna theory consists of two main theorems.

\begin{namedthm*}{First Main Theorem}\label{fmt} Let $f\colon\mathbb{C}\rightarrow \mathbb{P}^n(\mathbb{C})$ be a holomorphic curve and let $D$ be a hypersurface of degree $d$ in $\mathbb{P}^n(\mathbb{C})$ such that $f(\mathbb{C})\not\subset D$. Then for every $r>1$, the following holds
\[
m_f(r,D)
+
N_f(r,D)
\,
=
\,
d\,T_f(r)
+
O(1),
\]
hence
\begin{equation}
\label{-fmt-inequality}
N_f(r,D)
\,
\leq
\,
d\,T_f(r)+O(1).
\end{equation}
\end{namedthm*}

A holomorphic curve $f\colon\mathbb{C}\rightarrow \mathbb{P}^n(\mathbb{C})$ is said to be  \textsl{linearly nondegenerate} if its image is not contained in any hyperplane. For non-negatively valued functions $\varphi(r)$, $\psi(r)$, we write
\[
\varphi(r)
\,
\leq
\,
\psi(r)\parallel
\]
if this inequality holds outside a Borel subset $E$ of $(0,+\infty)$ of finite Lebesgue measure. Next is the Second Main Theorem of Cartan \cite{cartan1933}.
 
\begin{namedthm*}{Second Main Theorem}\label{smt}
Let $f\colon\mathbb{C}\rightarrow \mathbb{P}^n(\mathbb{C})$ and let $\{H_i\}_{1\leq i\leq q}$ be a family of hyperplanes in general position in $\mathbb{P}^n(\mathbb{C})$. Then the following estimate holds:
\[
(q-n-1)
\,
T_f(r)
\,
\leq
\,
\sum_{i=1}^q N_f^{[n]}(r,H_i)+S_f(r),
\]
where $S_f(r)$  is a small term compared with $T_f(r)$
\[
S_f(r)
\,
=
\,
o(T_f(r))\parallel.
\]
\end{namedthm*}

The next three theorems can be deduced from the Second Main Theorem.

\begin{thm}\label{diagonal}
Let $\{H_i\}_{1\leq i\leq n+2}$ be a family of hyperplanes in general position in $\mathbb{P}^n(\mathbb{C})$ with $n\geq 2$. If $f\colon\mathbb{C}\rightarrow \mathbb{P}^n(\mathbb{C})\setminus \cup_{i=1}^{n+2}H_i$ is an entire curve, then its image lies in one of the diagonal hyperplanes of $\{H_i\}_{1\leq i\leq n+2}$.
\end{thm}

The following strengthened theorem is due to Dufresnoy \cite{dufresnoy1944}.

\begin{thm}
\label{diagonal_streng}
If a holomorphic map $f\colon\mathbb{C}\rightarrow\mathbb{P}^n(\mathbb{C})$ has its image in the complement of $n + p$ hyperplanes $H_1,\ldots, H_{n+p}$ in general position, then this image is contained in a linear subspace of dimension $\floor*{\frac{n}{p}}$.
\end{thm}

As a consequence, we have the classical generalization of Picard's Theorem (case $n=1$), due to Fujimoto \cite{Fujimoto1972} (see also \cite{Green1972}).

\begin{thm}
\label{fujimoto-green}
The complement of a collection of $2n+1$ hyperplanes in general position in $\mathbb{P}^n(\mathbb{C})$ is hyperbolic.
\end{thm}

For hyperplanes that are not in general position, we have the following result (see \cite{kobayashi1998}, Theorem~3.10.15).

\begin{thm}
\label{-hyperplane-not-in-general-position}
Let $\{H_i\}_{1\leq i\leq q}$ be a family of $q\geq 3$ hyperplanes that are not in general position in $\mathbb{P}^n(\mathbb{C})$. If $f\colon\mathbb{C}\rightarrow \mathbb{P}^n(\mathbb{C})\setminus \cup_{i=1}^{q}H_i$ is an entire curve, then its image lies in some hyperplane.
\end{thm}

\section{Starting lemmas}
\label{-starting-lemmas}

Let us introduce some notations before going to other applications. Let $\{H_i\}_{1\leq i\leq q}$ be a family of generic hyperplanes of $\mathbb{P}^n(\mathbb{C})$, where $H_i=\{h_i=0\}$. For some integer $0\leq k\leq n-1$ and some subset $I_k=\{i_1,\dots,i_{n-k}\}$ of the index set $\{1,\ldots,q\}$ having cardinality $n-k$, the linear subspace $P_{k,I_k}= \cap_{i\in I_k}H_i\simeq\mathbb{P}^k(\mathbb{C})$ is called a \textsl{subspace of dimension $k$}. For a holomorphic mapping $f\colon\mathbb{C}\rightarrow\mathbb{P}^n(\mathbb{C})$, we define
\begin{equation*}
n_f(t,P_{k,I_k})
\,
:=
\,
\sum_{|z|<t,f(z)\in P_{k,I_k}}
\min_{i\in I_k}
\,
\ord_z (h_i\circ f)
\quad\quad {\scriptstyle (t\,>\,0)},
\end{equation*}
where we take the sum only for $z$ in the preimage of $P_{k,I_k}$, and
\begin{equation}
\label{counting function for subspace}
N_f(r,P_{k,I_k})
\,
:=
\,
\int_1^r\dfrac{n_f(t,P_{k,I_k})}{t}\,\der t
\quad\quad {\scriptstyle (r\,>\,1)}.
\end{equation}

We denote by $P_{k,I_k}^*$ the complement $P_{k,I_k}\setminus\big(\cup_{i\not\in I_k} H_i\big)$ which will be called a \textsl{star-subspace of dimension $k$}. We can also define $n_f(t,P_{k,I_k}^*)$ and $N_f(r,P_{k,I_k}^*)$. Assume now $q=2n+1+m$ with $m\geq 0$. Consider complements of the form
\begin{equation}\label{hyperbolic_complement}
\mathbb{P}^n(\mathbb{C})\setminus(\cup_{i=1}^{2n+1+m}H_i\setminus A_{m,n}),
\end{equation}
where $A_{m,n}$ is a set of at most $m$ elements of the form $P_{k,I_k}^*$ ($0\leq k\leq n-2$). We note that if $m=0$, these complements are hyperbolic by Theorem~\ref{fujimoto-green}.

In $\mathbb{P}^2(\mathbb{C})$, a union of lines $\cup_{i=1}^q H_i$ is in general position if any three lines have empty intersection, and it is generic if in addition any three intersection points between three distinct pairs of lines are not collinear. 

\begin{lem}\label{start_lem_P2} 
In $\mathbb{P}^2(\mathbb{C})$, if $m\leq 3$, all complements of the form \eqref{hyperbolic_complement} are hyperbolic.
\end{lem}


\begin{proof}
Without loss of generality, we can assume that $A_{m,2}$ is a set of $m$ distinct points belonging to $\cup_{1\leq i_1<i_2\leq 5+m}\,H_{i_1}\cap H_{i_2}$.

When $m=1$, an entire curve $f\colon\mathbb{C}\rightarrow \mathbb{P}^2(\mathbb{C})\setminus (\cup_{i=1}^6H_i\setminus A_{1,2})$, if it exists, must avoid at least four lines.
\begin{center}
\begin{picture}(0,0)%
\includegraphics{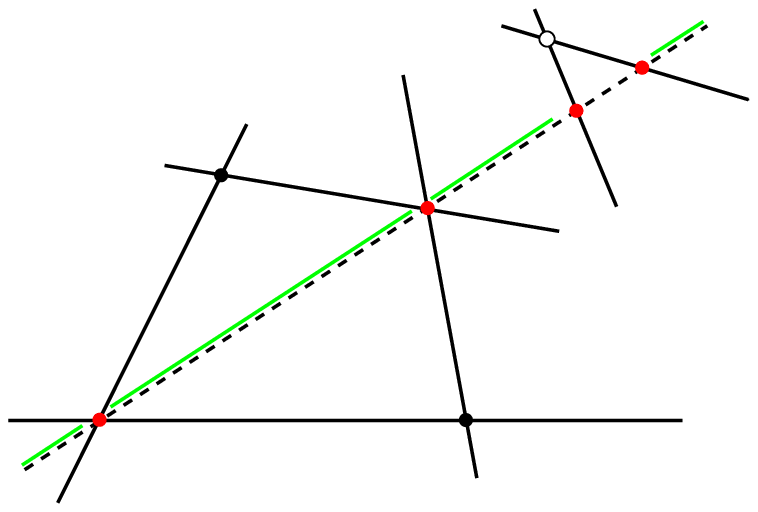}%
\end{picture}%
\setlength{\unitlength}{4144sp}%
\begingroup\makeatletter\ifx\SetFigFont\undefined%
\gdef\SetFigFont#1#2#3#4#5{%
  \reset@font\fontsize{#1}{#2pt}%
  \fontfamily{#3}\fontseries{#4}\fontshape{#5}%
  \selectfont}%
\fi\endgroup%
\begin{picture}(3427,2325)(471,-2379)
\put(3011,-201){\makebox(0,0)[lb]{\smash{{\SetFigFont{11}{13.2}{\familydefault}{\mddefault}{\updefault}{\color[rgb]{0,0,0}$A_{1,2}$}%
}}}}
\put(1349,-1389){\makebox(0,0)[lb]{\smash{{\SetFigFont{11}{13.2}{\familydefault}{\mddefault}{\updefault}{\color[rgb]{0,0,0}$f(\mathbb{C})$}%
}}}}
\end{picture}%
\end{center}

\noindent
By Theorem \ref{diagonal}, its image lies in a diagonal line, 
which does not contain the intersection point of the two remaining lines by the generic condition. Hence, $f$ must be contained in the complement of four points in a line. By Picard's theorem, $f$ is constant, which is contradiction.

When $m=2$, $A_{2,2}$ is a set consisting of two points $A$, $B$, where $A=H_{i_1}\cap H_{i_2}$, $B=H_{i_3}\cap H_{i_4}$. We denote by $I$ the index set $\{i_1,i_2,i_3,i_4\}$, which has three elements if both $A$ and $B$ belong to a single line $H_i$ and which has four elements otherwise.
\begin{center}
\begin{picture}(0,0)%
\includegraphics{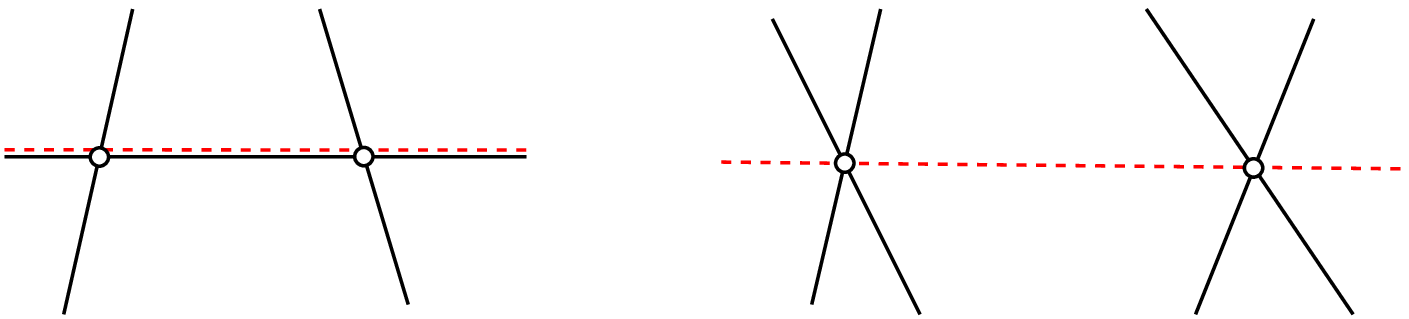}%
\end{picture}%
\setlength{\unitlength}{4144sp}%
\begingroup\makeatletter\ifx\SetFigFont\undefined%
\gdef\SetFigFont#1#2#3#4#5{%
  \reset@font\fontsize{#1}{#2pt}%
  \fontfamily{#3}\fontseries{#4}\fontshape{#5}%
  \selectfont}%
\fi\endgroup%
\begin{picture}(6425,1439)(564,-1838)
\put(5611,-619){\makebox(0,0)[lb]{\smash{{\SetFigFont{12}{14.4}{\familydefault}{\mddefault}{\updefault}{\color[rgb]{0,0,0}$H_{i_3}$}%
}}}}
\put(6588,-612){\makebox(0,0)[lb]{\smash{{\SetFigFont{12}{14.4}{\familydefault}{\mddefault}{\updefault}{\color[rgb]{0,0,0}$H_{i_4}$}%
}}}}
\put(4622,-577){\makebox(0,0)[lb]{\smash{{\SetFigFont{12}{14.4}{\familydefault}{\mddefault}{\updefault}{\color[rgb]{0,0,0}$H_{i_2}$}%
}}}}
\put(2602,-1009){\makebox(0,0)[lb]{\smash{{\SetFigFont{12}{14.4}{\familydefault}{\mddefault}{\updefault}{\color[rgb]{0,0,0}\red{$\mathcal{L}$}}%
}}}}
\put(5299,-1079){\makebox(0,0)[lb]{\smash{{\SetFigFont{12}{14.4}{\familydefault}{\mddefault}{\updefault}{\color[rgb]{0,0,0}\red{$\mathcal{L}$}}%
}}}}
\put(800,-625){\makebox(0,0)[lb]{\smash{{\SetFigFont{12}{14.4}{\familydefault}{\mddefault}{\updefault}{\color[rgb]{0,0,0}$H_{i_1}$}%
}}}}
\put(790,-1277){\makebox(0,0)[lb]{\smash{{\SetFigFont{12}{14.4}{\familydefault}{\mddefault}{\updefault}{\color[rgb]{0,0,0}$A$}%
}}}}
\put(1302,-1275){\makebox(0,0)[lb]{\smash{{\SetFigFont{12}{14.4}{\familydefault}{\mddefault}{\updefault}{\color[rgb]{0,0,0}$H_{i_2}=H_{i_3}$}%
}}}}
\put(2340,-1266){\makebox(0,0)[lb]{\smash{{\SetFigFont{12}{14.4}{\familydefault}{\mddefault}{\updefault}{\color[rgb]{0,0,0}$B$}%
}}}}
\put(3861,-580){\makebox(0,0)[lb]{\smash{{\SetFigFont{12}{14.4}{\familydefault}{\mddefault}{\updefault}{\color[rgb]{0,0,0}$H_{i_1}$}%
}}}}
\put(4195,-1313){\makebox(0,0)[lb]{\smash{{\SetFigFont{12}{14.4}{\familydefault}{\mddefault}{\updefault}{\color[rgb]{0,0,0}$A$}%
}}}}
\put(6470,-1328){\makebox(0,0)[lb]{\smash{{\SetFigFont{12}{14.4}{\familydefault}{\mddefault}{\updefault}{\color[rgb]{0,0,0}$B$}%
}}}}
\put(2134,-595){\makebox(0,0)[lb]{\smash{{\SetFigFont{12}{14.4}{\familydefault}{\mddefault}{\updefault}{\color[rgb]{0,0,0}$H_{i_4}$}%
}}}}
\end{picture}%
\end{center}

\noindent
Let $f\colon\mathbb{C}\rightarrow \mathbb{P}^2(\mathbb{C})\setminus (\cup_{i=1}^7H_i\setminus A_{2,2})$ be an entire curve. If $z\in f^{-1}(A)$, we have 
\begin{align*}
\ord_z(h_{i_1}\circ f)
&
\,
\geq
\,
1,\\
\ord_z(h_{i_2}\circ f)
&
\,
\geq
\,1.
\end{align*}
This implies 
 \begin{equation}\label{es at point A,z}
\min
\,
\{\ord_z (h_{i_1}\circ f),2\}
+
\min
\,
\{\ord_z (h_{i_2}\circ f),2\}
\,
\leq
\,
3
\,
\min_{1\leq j\leq 2}
\,
\ord_z (h_{i_j}\circ f).
\end{equation}
Hence, by summing this inequality
\begin{equation}
\label{-estimate-at-the-point-A}
\sum_{|z|<t,f(z)=A}
\min
\,
\{\ord_z (h_{i_1}\circ f),2\}
+
\sum_{|z|<t,f(z)=A}
\min
\,
\{\ord_z (h_{i_2}\circ f),2\}
\,
\leq
\,
3
\sum_{|z|<t,f(z)=A}
\min_{1\leq j\leq 2}
\,
\ord_z (h_{i_j}\circ f).
\end{equation}
Similarly for $h_{i_3}$, $h_{i_4}$ and $z\in f^{-1}(B)$, we have
\begin{equation}
\label{-estimate-at-the-point-B}
\sum_{|z|<t,f(z)=B}
\min
\,
\{\ord_z (h_{i_3}\circ f),2\}
+
\sum_{|z|<t,f(z)=B}
\min
\,
\{\ord_z (h_{i_4}\circ f),2\}
\,
\leq
\,
3 \sum_{|z|<t,f(z)=B}
\min_{3\leq j\leq 4}
\,
\ord_z (h_{i_j}\circ f).
\end{equation}
By taking the sum of both sides of these inequalities and by integrating, we obtain
\begin{equation}\label{-compare-the-counting-functions1}
\sum_{i\in I} N_f^{[2]}(r,H_i)
\,
\leq
\,
3
\,
\big(N_f(r,A)+N_f(r,B)\big).
\end{equation}

Now, let $\mathcal{L}=\{\ell=0\}$ be the line passing through $A$ and $B$. Since $\ell=\alpha_1 h_{i_1}+\alpha_2 h_{i_2}=\alpha_3 h_{i_3}+\alpha_4 h_{i_4}$ for some $\alpha_1,\ldots,\alpha_4\in\mathbb{C}$, the following inequalities hold
\begin{align}
\label{-compare-the-order_at-each-point}
\min_{1\leq j\leq 2}
\,
\ord_z (h_{i_j}\circ f)
&
\,
\leq
\,
\ord_z(\ell\circ f)
\quad\quad {\scriptstyle (z\,\in\,f^{-1}(A))},\notag\\
\min_{3\leq j\leq 4}
\,
\ord_z (h_{i_j}\circ f)
&
\,
\leq
\,
\ord_z(\ell\circ f)
\quad\quad {\scriptstyle (z\,\in\,f^{-1}(B))}.
\end{align}
Since $f^{-1}(A)$ and $f^{-1}(B)$ are two disjoint subsets of $f^{-1}(\mathcal{L})$, by taking the sum of both sides of these inequalities on discs and by integrating, we obtain
\begin{equation}
\label{-compare-the-counting-function-at-points-and-line-L}
N_f(r,A)
+
N_f(r,B)
\,
\leq
\,
N_f(r,\mathcal{L}).
\end{equation}

If $f$ would be linearly nondegenerate, then starting from Cartan Second Main Theorem, and using \eqref{-compare-the-counting-functions1}, \eqref{-compare-the-counting-function-at-points-and-line-L}, we would get
\begin{align}\label{-last-estimate-conclude1}
4\,T_f(r)
&
\,
\leq
\,
\sum_{i=1}^{7}
N_f^{[2]}(r,H_i)
+
S_f(r)\notag\\
&
\,
=
\,
\sum_{i\in I}N_f^{[2]}(r,H_i)
+
S_f(r)\notag\\
&
\,
\leq
\,
3\,\big(N_f(r,A)+N_f(r,B)\big)
+
S_f(r)\notag\\
&
\,
\leq
\,
3\,N_f(r,\mathcal{L})
+
S_f(r)\notag\\
\explain{Use \eqref{-fmt-inequality}}
\ \ \ \ \ \ \ \ \ \ \ \ \
&
\,
\leq
\,
3\,T_f(r)
+
S_f(r),
\end{align}
which is absurd. Thus, any entire curve $f:\mathbb{C}\rightarrow \mathbb{P}^2(\mathbb{C})\setminus (\cup_{i=1}^7H_i\setminus A_{2,2})$ must be contained in some line $L$. Furthermore, the number of points of intersection between $L$ and $\cup_{i=1}^7H_i\setminus\{A,B\}$ is at least 3 by the generic condition. By Picard's Theorem, this contradicts the assumption that $f$ is nonconstant.

When $m=3$, $A_{3,2}$ is a set consisting of three points $A$, $B$, $C$, where $A=H_{i_1}\cap H_{i_2}$, $B=H_{i_3}\cap H_{i_4}$, $C=H_{i_5}\cap H_{i_6}$. In this case, the index set $J=\{i_1,i_2,i_3,i_4,i_5,i_6\}$ may contain $4$, $5$ or $6$ elements.

\begin{center}
\begin{picture}(0,0)%
\includegraphics{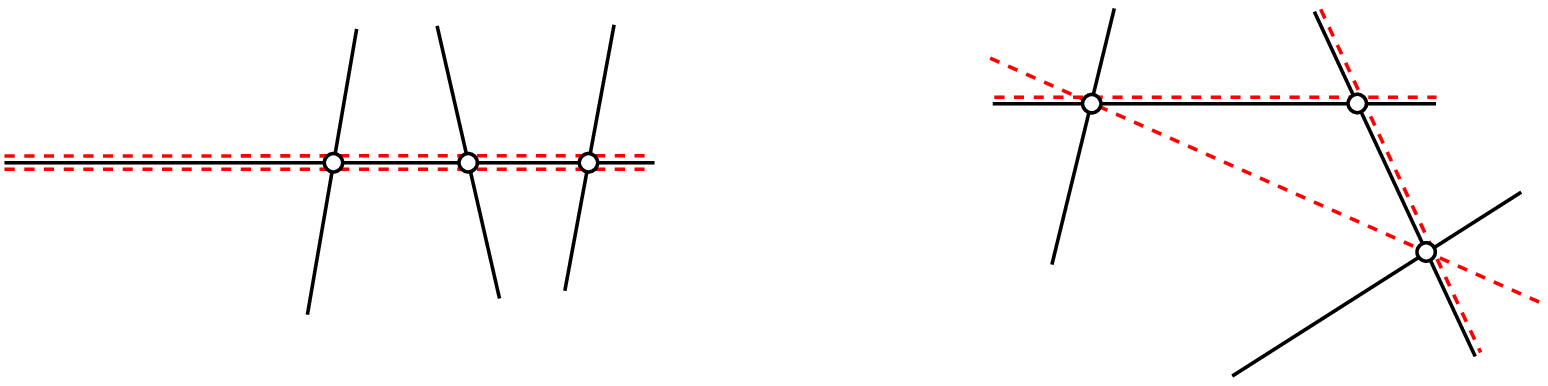}%
\end{picture}%
\setlength{\unitlength}{4144sp}%
\begingroup\makeatletter\ifx\SetFigFont\undefined%
\gdef\SetFigFont#1#2#3#4#5{%
  \reset@font\fontsize{#1}{#2pt}%
  \fontfamily{#3}\fontseries{#4}\fontshape{#5}%
  \selectfont}%
\fi\endgroup%
\begin{picture}(7082,1840)(261,-2124)
\put(3097,-575){\makebox(0,0)[lb]{\smash{{\SetFigFont{12}{14.4}{\familydefault}{\mddefault}{\updefault}{\color[rgb]{0,0,0}$H_{i_6}$}%
}}}}
\put(5035,-441){\makebox(0,0)[lb]{\smash{{\SetFigFont{12}{14.4}{\familydefault}{\mddefault}{\updefault}{\color[rgb]{0,0,0}$H_{i_2}$}%
}}}}
\put(6658,-1717){\makebox(0,0)[lb]{\smash{{\SetFigFont{12}{14.4}{\familydefault}{\mddefault}{\updefault}{\color[rgb]{0,0,0}$C$}%
}}}}
\put(6265,-964){\makebox(0,0)[lb]{\smash{{\SetFigFont{12}{14.4}{\familydefault}{\mddefault}{\updefault}{\color[rgb]{0,0,0}$B$}%
}}}}
\put(5472,-691){\makebox(0,0)[lb]{\smash{{\SetFigFont{12}{14.4}{\familydefault}{\mddefault}{\updefault}{\color[rgb]{0,0,0}$H_{i_1}= H_{i_3}$}%
}}}}
\put(6383,-431){\makebox(0,0)[lb]{\smash{{\SetFigFont{12}{14.4}{\familydefault}{\mddefault}{\updefault}{\color[rgb]{0,0,0}$H_{i_4}=H_{i_5}$}%
}}}}
\put(2816,-952){\makebox(0,0)[lb]{\smash{{\SetFigFont{12}{14.4}{\familydefault}{\mddefault}{\updefault}{\color[rgb]{0,0,0}$C$}%
}}}}
\put(327,-1295){\makebox(0,0)[lb]{\smash{{\SetFigFont{12}{14.4}{\familydefault}{\mddefault}{\updefault}{\color[rgb]{0,0,0}$H_{i_1}= H_{i_3}= H_{i_5}$}%
}}}}
\put(1622,-959){\makebox(0,0)[lb]{\smash{{\SetFigFont{12}{14.4}{\familydefault}{\mddefault}{\updefault}{\color[rgb]{0,0,0}$A$}%
}}}}
\put(2174,-958){\makebox(0,0)[lb]{\smash{{\SetFigFont{12}{14.4}{\familydefault}{\mddefault}{\updefault}{\color[rgb]{0,0,0}$B$}%
}}}}
\put(2340,-518){\makebox(0,0)[lb]{\smash{{\SetFigFont{12}{14.4}{\familydefault}{\mddefault}{\updefault}{\color[rgb]{0,0,0}$H_{i_4}$}%
}}}}
\put(1556,-532){\makebox(0,0)[lb]{\smash{{\SetFigFont{12}{14.4}{\familydefault}{\mddefault}{\updefault}{\color[rgb]{0,0,0}$H_{i_2}$}%
}}}}
\put(6136,-2060){\makebox(0,0)[lb]{\smash{{\SetFigFont{12}{14.4}{\familydefault}{\mddefault}{\updefault}{\color[rgb]{0,0,0}$H_{i_6}$}%
}}}}
\put(5033,-983){\makebox(0,0)[lb]{\smash{{\SetFigFont{12}{14.4}{\familydefault}{\mddefault}{\updefault}{\color[rgb]{0,0,0}$A$}%
}}}}
\end{picture}%
\end{center}

\noindent

\begin{center}
\begin{picture}(0,0)%
\includegraphics{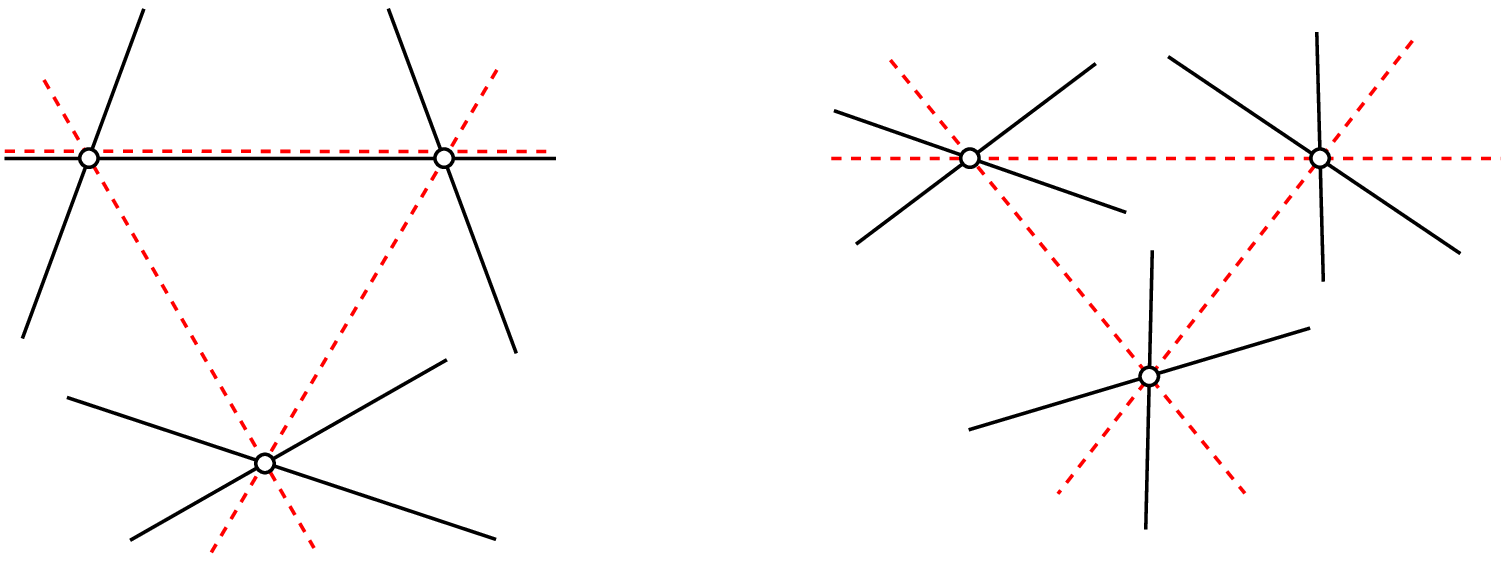}%
\end{picture}%
\setlength{\unitlength}{4144sp}%
\begingroup\makeatletter\ifx\SetFigFont\undefined%
\gdef\SetFigFont#1#2#3#4#5{%
  \reset@font\fontsize{#1}{#2pt}%
  \fontfamily{#3}\fontseries{#4}\fontshape{#5}%
  \selectfont}%
\fi\endgroup%
\begin{picture}(6884,2645)(699,-2284)
\put(1001,-2121){\makebox(0,0)[lb]{\smash{{\SetFigFont{12}{14.4}{\familydefault}{\mddefault}{\updefault}{\color[rgb]{0,0,0}$H_{i_6}$}%
}}}}
\put(5089,-297){\makebox(0,0)[lb]{\smash{{\SetFigFont{12}{14.4}{\familydefault}{\mddefault}{\updefault}{\color[rgb]{0,0,0}$A$}%
}}}}
\put(6100,-1529){\makebox(0,0)[lb]{\smash{{\SetFigFont{12}{14.4}{\familydefault}{\mddefault}{\updefault}{\color[rgb]{0,0,0}$C$}%
}}}}
\put(4928,-1556){\makebox(0,0)[lb]{\smash{{\SetFigFont{12}{14.4}{\familydefault}{\mddefault}{\updefault}{\color[rgb]{0,0,0}$H_{i_5}$}%
}}}}
\put(2406,-2220){\makebox(0,0)[lb]{\smash{{\SetFigFont{12}{14.4}{\familydefault}{\mddefault}{\updefault}{\color[rgb]{0,0,0}$H_{i_5}$}%
}}}}
\put(848,-610){\makebox(0,0)[lb]{\smash{{\SetFigFont{12}{14.4}{\familydefault}{\mddefault}{\updefault}{\color[rgb]{0,0,0}$A$}%
}}}}
\put(1455,-627){\makebox(0,0)[lb]{\smash{{\SetFigFont{12}{14.4}{\familydefault}{\mddefault}{\updefault}{\color[rgb]{0,0,0}$H_{i_1}= H_{i_3}$}%
}}}}
\put(2838,-609){\makebox(0,0)[lb]{\smash{{\SetFigFont{12}{14.4}{\familydefault}{\mddefault}{\updefault}{\color[rgb]{0,0,0}$B$}%
}}}}
\put(1859,-1642){\makebox(0,0)[lb]{\smash{{\SetFigFont{12}{14.4}{\familydefault}{\mddefault}{\updefault}{\color[rgb]{0,0,0}$C$}%
}}}}
\put(1411,187){\makebox(0,0)[lb]{\smash{{\SetFigFont{12}{14.4}{\familydefault}{\mddefault}{\updefault}{\color[rgb]{0,0,0}$H_{i_2}$}%
}}}}
\put(2177,214){\makebox(0,0)[lb]{\smash{{\SetFigFont{12}{14.4}{\familydefault}{\mddefault}{\updefault}{\color[rgb]{0,0,0}$H_{i_4}$}%
}}}}
\put(5632,-2143){\makebox(0,0)[lb]{\smash{{\SetFigFont{12}{14.4}{\familydefault}{\mddefault}{\updefault}{\color[rgb]{0,0,0}$H_{i_6}$}%
}}}}
\put(4312,-112){\makebox(0,0)[lb]{\smash{{\SetFigFont{12}{14.4}{\familydefault}{\mddefault}{\updefault}{\color[rgb]{0,0,0}$H_{i_2}$}%
}}}}
\put(4352,-726){\makebox(0,0)[lb]{\smash{{\SetFigFont{12}{14.4}{\familydefault}{\mddefault}{\updefault}{\color[rgb]{0,0,0}$H_{i_1}$}%
}}}}
\put(6883,-373){\makebox(0,0)[lb]{\smash{{\SetFigFont{12}{14.4}{\familydefault}{\mddefault}{\updefault}{\color[rgb]{0,0,0}$B$}%
}}}}
\put(6108, 47){\makebox(0,0)[lb]{\smash{{\SetFigFont{12}{14.4}{\familydefault}{\mddefault}{\updefault}{\color[rgb]{0,0,0}$H_{i_3}$}%
}}}}
\put(6769,143){\makebox(0,0)[lb]{\smash{{\SetFigFont{12}{14.4}{\familydefault}{\mddefault}{\updefault}{\color[rgb]{0,0,0}$H_{i_4}$}%
}}}}
\end{picture}%
\end{center}

Suppose to the contrary that there is an entire curve $f\colon \mathbb{C}\rightarrow \mathbb{P}^2(\mathbb{C})\setminus(\cup_{i=1}^8H_i\setminus A_{3,2})$. Similarly as above, cf. \eqref{-estimate-at-the-point-A}, \eqref{-estimate-at-the-point-B}, \eqref{-compare-the-counting-functions1}, we can show in all the four illustrated cases
\begin{equation*}\label{compare the counting functions2}
\sum_{i\in J}
 N_f^{[2]}(r,H_i)
\,
\leq
\,
3
\,
\big(
N_f(r,A)
+
N_f(r,B)
+
N_f(r,C)
\big).
\end{equation*}

Next, let $\mathcal{C}=\{\mathsf{c}=0\}$ be the degenerate cubic consisting of the three lines $AB=\{\ell_{AB}=0\}$, $BC=\{\ell_{BC}=0\}$, and $CA=\{\ell_{CA}=0\}$. Similarly as in \eqref{-compare-the-order_at-each-point}, we have
\begin{align*}
2
\,
\min_{1\leq j\leq 2}
\,
\ord_z (h_{i_j}\circ f)
&
\,
\leq
\,
\ord_z(\ell_{AB}\circ f)
+
\ord_z(\ell_{CA}\circ f)\\
&
\,
=
\,\ord_z(\mathsf{c}\circ f)
\qquad\qquad \qquad\qquad\quad\quad{\scriptstyle (z\,\in\,f^{-1}(A))}.
\end{align*}
We also have two other inequalities for $h_{i_3}$, $h_{i_4}$, $z\in f^{-1}(B)$ and for $h_{i_5}$, $h_{i_6}$, $z\in f^{-1 }(C)$. Summing these inequalities and integrating, we get
\begin{equation*}
\label{-compare-the-counting-function-at-points-and-cubic-C}
2
\,
\big(
N_f(r,A)
+
N_f(r,B)
+
N_f(r,C)
\big)
\,
\leq
\,
N_f(r,\mathcal{C}).
\end{equation*}

If the curve $f$ is linearly nondegenerate, then by proceeding as we did in \eqref{-last-estimate-conclude1}, we also get a contradiction.
\begin{align*}
5\,T_f(r)
&
\,
\leq
\,
\sum_{i=1}^{8}
N_f^{[2]}(r,H_i)
+
S_f(r)\\
&
\,
\leq
\,
3
\,
\big(
N_f(r,A)
+
N_f(r,B)
+
N_f(r,C)
\big)
+
S_f(r)\\
&
\,
\leq 
\,
\dfrac{3}{2}
\,
N_f(r,\mathcal{C})
+
S_f(r)\\
&
\,
\leq
\,
\dfrac{9}{2}
\,T_f(r)
+
S_f(r).
\end{align*} 
Thus the curve $f$ must be contained in some line. By analyzing the position of this line with respect to $\{H_i\}_{1\leq i\leq 8}\setminus \{A,B,C\}$ and by using Picard's theorem, we conclude as above.
\end{proof}

In $\mathbb{P}^3(\mathbb{C})$, the generic condition for the family of planes $\{H_i\}_{1\leq i\leq q}$ excludes the following cases.
\begin{itemize}
\item[\textbf{(1)}] There are three disjoint subsets $I$, $J$, $K$ of $\{1,\ldots,q\}$ with $|I|=3$, $|J|=2$, $|K|=3$ such that the diagonal (hyper)plane $H_{IJ}$ contains the point $\cap_{k\in K}H_k$.
\item[\textbf{(2)}] There are four disjoint subsets $I$, $J$, $K_1$, $K_2$ of $\{1,\ldots,q\}$ with $|I|=3$, $|J|=2$, $|K_1|=|K_2|=2$ such that the three points $(\cap_{k_1\in K_1}H_{k_1})\cap H_{IJ}$, $(\cap_{k_2\in K_2}H_{k_2})\cap H_{IJ}$ and $\cap_{i\in I}H_i$ are collinear.
\end{itemize}

\begin{center}
\begin{picture}(0,0)%
\includegraphics{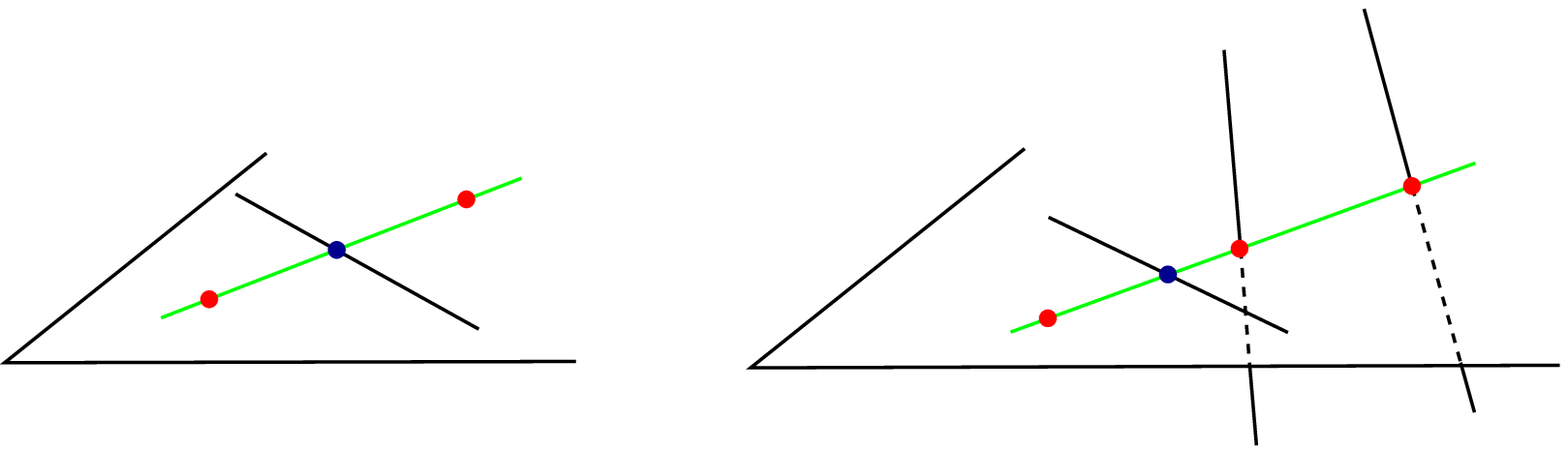}%
\end{picture}%
\setlength{\unitlength}{4144sp}%
\begingroup\makeatletter\ifx\SetFigFont\undefined%
\gdef\SetFigFont#1#2#3#4#5{%
  \reset@font\fontsize{#1}{#2pt}%
  \fontfamily{#3}\fontseries{#4}\fontshape{#5}%
  \selectfont}%
\fi\endgroup%
\begin{picture}(7385,2180)(15,-2312)
\put(987,-1726){\makebox(0,0)[lb]{\smash{{\SetFigFont{6}{7.2}{\familydefault}{\mddefault}{\updefault}{\color[rgb]{0,0,0}$\cap_{i\in I}H_i$}%
}}}}
\put(1190,-1096){\makebox(0,0)[lb]{\smash{{\SetFigFont{6}{7.2}{\familydefault}{\mddefault}{\updefault}{\color[rgb]{0,0,0}$\cap_{j\in J}H_j$}%
}}}}
\put(2187,-1253){\makebox(0,0)[lb]{\smash{{\SetFigFont{6}{7.2}{\familydefault}{\mddefault}{\updefault}{\color[rgb]{0,0,0}$\cap_{k\in K}H_k$}%
}}}}
\put(316,-1812){\makebox(0,0)[lb]{\smash{{\SetFigFont{6}{7.2}{\familydefault}{\mddefault}{\updefault}{\color[rgb]{0,0,0}$H_{IJ}$}%
}}}}
\put(3818,-1842){\makebox(0,0)[lb]{\smash{{\SetFigFont{6}{7.2}{\familydefault}{\mddefault}{\updefault}{\color[rgb]{0,0,0}$H_{IJ}$}%
}}}}
\put(5134,-444){\makebox(0,0)[lb]{\smash{{\SetFigFont{6}{7.2}{\familydefault}{\mddefault}{\updefault}{\color[rgb]{0,0,0}$\cap_{k_1\in K_1}H_{k_1}$}%
}}}}
\put(4966,-1805){\makebox(0,0)[lb]{\smash{{\SetFigFont{6}{7.2}{\familydefault}{\mddefault}{\updefault}{\color[rgb]{0,0,0}$\cap_{i\in I}H_i$}%
}}}}
\put(5777,-223){\makebox(0,0)[lb]{\smash{{\SetFigFont{6}{7.2}{\familydefault}{\mddefault}{\updefault}{\color[rgb]{0,0,0}$\cap_{k_2\in K_2}H_{k_2}$}%
}}}}
\put(4985,-1174){\makebox(0,0)[lb]{\smash{{\SetFigFont{6}{7.2}{\familydefault}{\mddefault}{\updefault}{\color[rgb]{0,0,0}$\cap_{j\in J}H_j$}%
}}}}
\end{picture}%
\end{center}

\begin{lem}\label{start_lem_P3}
 In $\mathbb{P}^3(\mathbb{C})$, if $m\leq 2$, all complements of the form \eqref{hyperbolic_complement} are hyperbolic.
\end{lem}

\begin{proof}
Without loss of generality, we can assume that $A_{m,3}$ is a set of $m$ elements belonging to:
\[
\big(
\cup_{1\leq i_1<i_2\leq 7+m}\,
(
H_{i_1}
\cap 
H_{i_2}
)^*
\big)\,
\bigcup\,
\big(
\cup_{1\leq i_1<i_2<i_3\leq 7+m}\,
H_{i_1}
\cap 
H_{i_2}
\cap H_{i_3}
\big).
\]
Suppose to the contrary that there exists a Brody curve $f\colon\mathbb{C}\rightarrow \mathbb{P}^3(\mathbb{C})\setminus(\cup_{i=1}^{7+m}H_i\setminus A_{m,3})$.

When $m=1$, the curve $f$ must avoid at least five planes. By Theorem~\ref{diagonal_streng}, its image is contained in some line $L$. By the generic condition, the number of intersection points between $L$ and $\cup_{i=1}^8 H_i\setminus A_{1,3}$ is at least $3$. 
By Picard's theorem, $f$ must be constant, which is a contradiction.

Next, we consider the case $m=2$. If $A_{2,3}=\{l_1^*,l_2^*\}$ where $l_1,l_2$ are lines, then the curve $f$ avoids five planes, say $\{H_i\}_{1\leq i\leq 5}$. By 
Theorems~\ref{diagonal_streng} and \ref{diagonal}, its image lands in some line $\mathcal{L}$, which is contained in a diagonal plane $\mathcal{P}$ of the family $\{H_i\}_{1\leq i\leq 5}$. We may assume that the plane $\mathcal{P}$ passes through the point $H_1\cap H_2\cap H_3$ and contains the line $H_4\cap H_5$. If the line $\mathcal{L}$ does not pass through the point $H_1\cap  H_2\cap H_3$, then it intersects $\{H_i\}_{1\leq i\leq 3}$ in three distinct points, hence $f$ is constant by Picard's theorem. Thus $\mathcal{L}$ must pass through the point $H_1\cap H_2\cap H_3$. In the plane $\mathcal{P}$, the curve $f$ can pass through the points $l_1\cap\mathcal{P}$ and $l_2\cap\mathcal{P}$. But by the generic condition, cf.~\textbf{(2)} above, the three points $H_1\cap H_2\cap H_3$, $l_1\cap\mathcal{P}$, $l_2\cap\mathcal{P}$ are not collinear. Hence, $f(\mathbb{C})$ is contained in a complement of at least three points in the line $\mathcal{L}$, which is impossible by Picard's theorem.

Two substantial cases remain:
\begin{itemize}
\smallskip\item[\textbf{(a)}] $A_{2,3}=\{A,l^*\}$, where $A$ is a point and $l$ is a line;
\smallskip\item[\textbf{(b)}] the set $A_{2,3}$ consists of two points.
\end{itemize}

We treat case \textbf{(a)}. If both $A$ and $l$ are contained in some common plane $H_i$, then $f$ avoids five planes. By Theorem~\ref{diagonal}, its image must be contained in some diagonal plane, which does not contain the point $A$ by the generic condition. Hence $f$ must avoid seven planes in general position, which is absurd by Theorem~\ref{fujimoto-green}. Thus, we can assume that $A=H_1\cap H_2\cap H_3$ and $l^*=(H_4\cap H_5)\setminus \cup_{i\not= 4,5} H_i$. Hence $f$ avoids four planes $H_i$ ($6\leq i\leq 9$).

First, we show that $f$ is linearly nondegenerate. Suppose to the contrary that $f(\mathbb{C})$ is contained in some plane $\mathcal{P}$. If $A\not\in\mathcal{P}$, then $f$ also avoids $H_1$, $H_2$, $H_3$, which is impossible by Theorem~\ref{fujimoto-green}. Hence the plane $\mathcal{P}$ must pass through the point $A$. If $f(\mathbb{C})$ is contained in some line $\mathcal{L}\subset\mathcal{P}$, then $\mathcal{L}$ must also pass through $A$, for the same reason. Note that the number of intersection points between $\mathcal{L}$ and $\{H_i\}_{6\leq i\leq 9}$ is at least $2$, and it equals $2$ only if either $\mathcal{L}$ passes through some point $H_{i_1}\cap H_{i_2}\cap H_{i_3}$ ($6\leq i_1<i_2<i_3\leq 9$) or $\mathcal{L}$ intersects two lines $H_{i_1}\cap H_{i_2}$, $H_{i_3}\cap H_{i_4}$ ($\{i_1,i_2,i_3,i_4\}=\{6,7,8,9\}$). If $\mathcal{L}$ has empty intersection with $H_4\cap H_5$, then $f$ avoids at least four points in the line $\mathcal{L}$, hence it is constant. If $\mathcal{L}$ intersects $H_4\cap H_5$, then by considering the diagonal plane passing through $A$ and containing $H_4\cap H_5$, the two cases where $|\mathcal{L}\cap \{H_i\}_{6\leq i\leq 9}|=2$ are excluded by the generic condition. Thus $f$ always avoids three distinct points in $\mathcal{L}$, hence it is constant.

Consequently, we can assume that $f$ does not land in any line in the plane $\mathcal{P}$. There are two possible positions of $\mathcal{P}$:
\begin{itemize}
\smallskip\item[\textbf{(a1)}] it is a diagonal plane containing $A$ and some line in $\cup_{6\leq i_1<i_2\leq 9}\,H_{i_1}\cap H_{i_2}$;
\smallskip\item[\textbf{(a2)}] it does not contain any line in $\cup_{6\leq i_1< i_2\leq 9}\, H_{i_1}\cap H_{i_2}$.
\end{itemize}

In case \textbf{(a1)}, assume that $\mathcal{P}$ contains the line $H_6\cap H_7$. Among $\{H_i\cap\mathcal{P}\}_{1\leq i\leq 9}$, two lines $H_6\cap \mathcal{P}$, $H_7\cap\mathcal{P}$ coincide, and dropping the line $H_1\cap\mathcal{P}$, by the generic condition, it remains seven lines $\{H_i\cap\mathcal{P}\}_{i\not=1,7}$ in general position in $\mathcal{P}$.
\begin{center}
\begin{picture}(0,0)%
\includegraphics{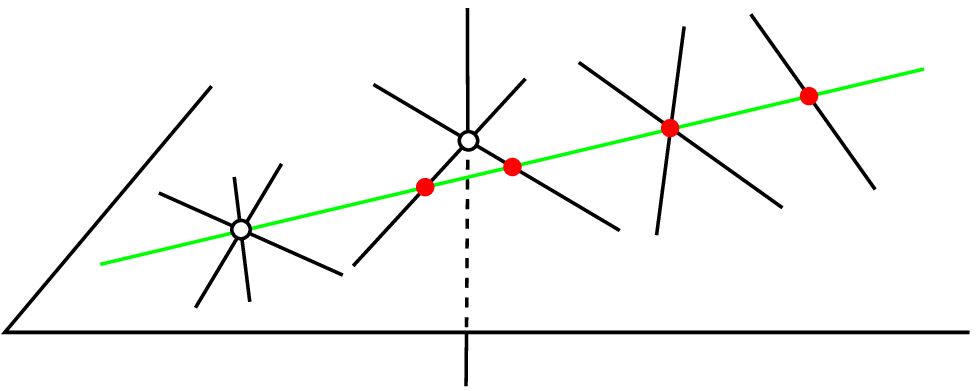}%
\end{picture}%
\setlength{\unitlength}{4144sp}%
\begingroup\makeatletter\ifx\SetFigFont\undefined%
\gdef\SetFigFont#1#2#3#4#5{%
  \reset@font\fontsize{#1}{#2pt}%
  \fontfamily{#3}\fontseries{#4}\fontshape{#5}%
  \selectfont}%
\fi\endgroup%
\begin{picture}(4454,1773)(609,-2444)
\put(3960,-1742){\makebox(0,0)[lb]{\smash{{\SetFigFont{10}{12.0}{\familydefault}{\mddefault}{\updefault}{\color[rgb]{0,0,0}$H_6\cap H_7$}%
}}}}
\put(2790,-2068){\makebox(0,0)[lb]{\smash{{\SetFigFont{10}{12.0}{\familydefault}{\mddefault}{\updefault}{\color[rgb]{0,0,0}$H_4\cap H_5$}%
}}}}
\put(1380,-1737){\makebox(0,0)[lb]{\smash{{\SetFigFont{10}{12.0}{\familydefault}{\mddefault}{\updefault}{\color[rgb]{0,0,0}$A$}%
}}}}
\put(2577,-1158){\makebox(0,0)[lb]{\smash{{\SetFigFont{10}{12.0}{\familydefault}{\mddefault}{\updefault}{\color[rgb]{0,0,0}$B$}%
}}}}
\put(844,-2108){\makebox(0,0)[lb]{\smash{{\SetFigFont{10}{12.0}{\familydefault}{\mddefault}{\updefault}{\color[rgb]{0,0,0}$\mathcal{P}$}%
}}}}
\end{picture}%
\end{center}

\noindent
Letting $B$ be the intersection point of the line $l=H_4\cap H_5$ with the plane $\mathcal{P}$, the curve $f$ lands in $\mathcal{P}\setminus (\cup_{1\leq i\leq 9}H_i\cap\mathcal{P})\setminus\{A,B\}))$. As in \eqref{-last-estimate-conclude1}, $f(\mathbb{C})$ is contained in some line, which is a contradiction.

Next, consider case \textbf{(a2)}.
\begin{center}
\begin{picture}(0,0)%
\includegraphics{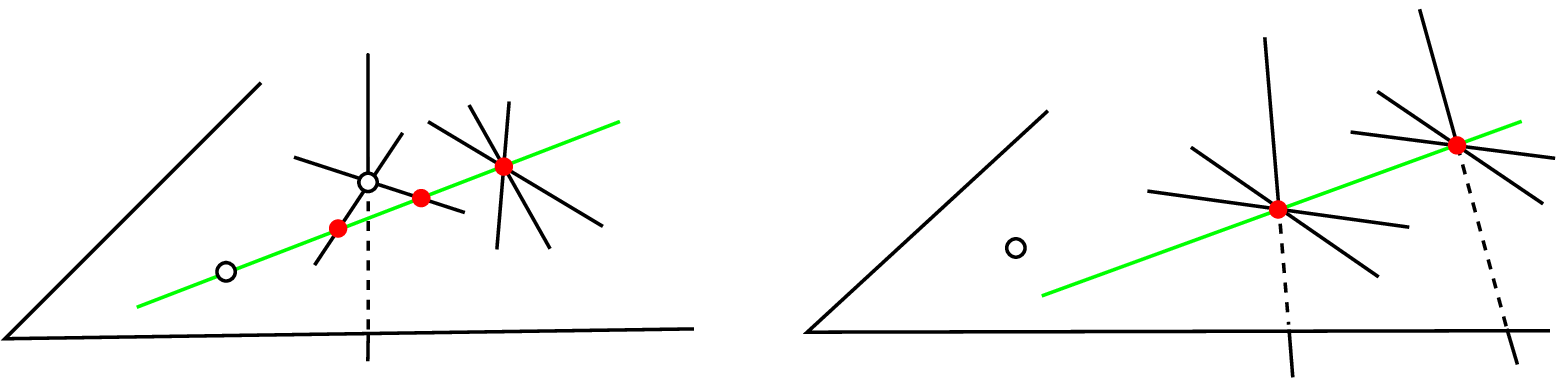}%
\end{picture}%
\setlength{\unitlength}{4144sp}%
\begingroup\makeatletter\ifx\SetFigFont\undefined%
\gdef\SetFigFont#1#2#3#4#5{%
  \reset@font\fontsize{#1}{#2pt}%
  \fontfamily{#3}\fontseries{#4}\fontshape{#5}%
  \selectfont}%
\fi\endgroup%
\begin{picture}(7131,1786)(24,-2149)
\put(887,-1626){\makebox(0,0)[lb]{\smash{{\SetFigFont{11}{13.2}{\familydefault}{\mddefault}{\updefault}{\color[rgb]{0,0,0}$A$}%
}}}}
\put(2538,-1234){\makebox(0,0)[lb]{\smash{{\SetFigFont{11}{13.2}{\familydefault}{\mddefault}{\updefault}{\color[rgb]{0,0,0}$H_6\cap H_7\cap H_8$}%
}}}}
\put(271,-1883){\makebox(0,0)[lb]{\smash{{\SetFigFont{11}{13.2}{\familydefault}{\mddefault}{\updefault}{\color[rgb]{0,0,0}$\mathcal{P}$}%
}}}}
\put(1742,-627){\makebox(0,0)[lb]{\smash{{\SetFigFont{11}{13.2}{\familydefault}{\mddefault}{\updefault}{\color[rgb]{0,0,0}$H_4\cap H_5$}%
}}}}
\put(6593,-510){\makebox(0,0)[lb]{\smash{{\SetFigFont{11}{13.2}{\familydefault}{\mddefault}{\updefault}{\color[rgb]{0,0,0}$H_8\cap H_9$}%
}}}}
\put(4641,-1448){\makebox(0,0)[lb]{\smash{{\SetFigFont{11}{13.2}{\familydefault}{\mddefault}{\updefault}{\color[rgb]{0,0,0}$A$}%
}}}}
\put(3992,-1859){\makebox(0,0)[lb]{\smash{{\SetFigFont{11}{13.2}{\familydefault}{\mddefault}{\updefault}{\color[rgb]{0,0,0}$\mathcal{P}$}%
}}}}
\put(5122,-619){\makebox(0,0)[lb]{\smash{{\SetFigFont{11}{13.2}{\familydefault}{\mddefault}{\updefault}{\color[rgb]{0,0,0}$H_6\cap H_7$}%
}}}}
\end{picture}%
\end{center}

\noindent
If $\mathcal{P}$ contains some point in $\cup_{6\leq i_1< i_2<i_3\leq 9}\, H_{i_1}\cap H_{i_2}\cap H_{i_3}$, say $H_6\cap H_7\cap H_8$, then the curve $f$ avoids three lines $H_i\cap\mathcal{P}$ ($6\leq i\leq 8$), which are not in general position. By Theorem~\ref{-hyperplane-not-in-general-position}, $f(\mathbb{C})$ must be contained in some line, which is a contradiction. Therefore, $\mathcal{P}$ does not contain any point in $\cup_{6\leq i_1< i_2<i_3\leq 9}\, H_{i_1}\cap H_{i_2}\cap H_{i_3}$. But then the curve $f$ avoids a collection of four lines $\{H_i\cap\mathcal{P}\}_{6\leq i\leq 9}$, which are in general position. By Theorem~\ref{diagonal}, its image must land in some diagonal line of this family, which is a contradiction.

Still in case \textbf{(a)}, we can therefore assume that $f$ is linearly nondegenerate. Assume that the omitted planes $H_6$, $H_7$, $H_8$, $H_9$ are given
in the homogeneous coordinates $[z_0:z_1:z_2:z_3]$ by equations $\{z_i=0\}$ ($0\leq i\leq 3$). By Theorems~\ref{omit-hyperplanes}, $f$ has a reduced representation of the form
\begin{equation}
\label{form of f_case1}
[1:e^{\lambda_1\,z+\mu_1}:e^{\lambda_2\,z+\mu_2}:e^{\lambda_3\,z+\mu_3}],
\end{equation}
where $\lambda_i$, $\mu_i$ are constants with $\lambda_i\not=0$ ($1\leq i\leq 3$ and $\lambda_i\not=\lambda_j$ ($i\not=j$). Let $\mathcal{D}$ be the diagonal plane passing through the point $A=H_1\cap H_2\cap H_3$ and containing the line $l=H_4\cap H_5$. By similar arguments as in Lemma~\ref{start_lem_P2}, cf. \eqref{-compare-the-order_at-each-point}, \eqref{-compare-the-counting-function-at-points-and-line-L}, we can show that
\begin{equation}
\label{-estimate-between-A-l_alpha}
N_f(r,A)
+
N_f(r,l^*)
\,
\leq
\,
N_f(r,\mathcal{D}).
\end{equation}
From the elementary inequality
\[
\min
\,
\{\ord_z(h_4\circ f),3\}
+
\min
\,
\{\ord_z(h_5\circ f),3\}
\,
\leq
\,
4
\min_{4\leq i\leq 5}
\,
\ord_z(h_i\circ f)\quad\quad {\scriptstyle(z\,\in\,f^{-1}(l^*))},
\]
by taking the sum on disks and then by integrating, we get
\begin{equation}
\label{-estimate-at-line-l}
N_f^{[3]}(r,H_4)
+
N_f^{[3]}(r,H_5)
\,
\leq
\,
4
\,
N_f(r,l^*).
\end{equation}

Next, we try to bound $N_f^{[3]}(r,H_i)$ 
($1\leq i\leq 3$) from above in terms of $N_f(r,A)$. Since $f$ is of the form \eqref{form of f_case1}, for any $z_1$, $z_2$ $\in f^{-1}(A)$, we have
\[
f^{(k)}(z_1)
\,
=
\,f^{(k)}(z_2)\quad\quad
{\scriptstyle (k\,\in\,\mathbb{N})},
\]
hence
\begin{equation}
\label{-ord-does-not-change}
\ord_{z_1}(h_i\circ f)
\,
=
\,
\ord_{z_2}(h_i\circ f)
\quad\quad 
{\scriptstyle (1\,\leq\,i\,\leq\,3)}.
\end{equation}
Thus, it suffices to consider the two cases:
\begin{itemize}
\smallskip \item[\textbf{(a3)}] $\ord_z (h_i\circ f)\leq 2$ for all $1\leq i\leq 3$ and for all $z\in f^{-1}(A)$;
\smallskip \item[\textbf{(a4)}] $\ord_z(h_i\circ f)\geq 3$ for some $i$ with $1\leq i\leq 3$ and for all $z\in f^{-1}(A)$.
\end{itemize}

In case \textbf{(a3)}, the elementary inequality 
\[
\sum_{i=1}^3
\min
\,
\{\ord_z(h_i\circ f),3\}
\,
\leq
\,
5
\min_{1\leq i\leq 3}
\,
\ord_z(h_i\circ f)\quad\quad
{\scriptstyle(z\,\in\,f^{-1}(A))},
\]
yields
\begin{equation}
\label{-estimate-at-point-A-case-1}
N_f^{[3]}(r,H_1)
+
N_f^{[3]}(r,H_2)
+
N_f^{[3]}(r,H_3)
\,
\leq
\,
5\,N_f(r,A).
\end{equation}
Since $f$ is linearly nondegenerate, we can proceed similarly as in \eqref{-last-estimate-conclude1}
\begin{align}\label{-smt-m=2-case1.1}
5\,T_f(r)
&
\,
\leq
\,
\sum_{i=1}^{9}
N_f^{[3]}(r,H_i)
+
S_f(r)\notag\\
&
\,
\leq
\,
5\,N_f(r,A)
+
4\,N_f(r,l^*)
+
S_f(r)\notag\\
&
\,
=
\,
5
\,
\big(
N_f(r,A)
+
N_f(r,l^*)\big)
-
N_f(r,l^*)
+S_f(r)\notag\\
&\,
\leq
\,
5\,N_f(r,\mathcal{D})
-
N_f(r,l^*)
+
S_f(r)\notag\\
&
\,
\leq
\,
5\,T_f(r)
-
N_f(r,l^*)
+
S_f(r).
\end{align}
This implies
\[
N_f(r,l^*)
=
S_f(r)
\]
and hence, by \eqref{-estimate-at-line-l}, we have
\[
N_f^{[3]}(r,H_4)
+
N_f^{[3]}(r,H_5)
\,
=
\,
S_f(r).
\]
Therefore, the first inequality of \eqref{-smt-m=2-case1.1} can be rewritten as
\begin{equation*}
5
\,
T_f(r)
\,
\leq
\,
\sum_{i=1}^3
N_f^{[3]}(r,H_i)
+
S_f(r).
\end{equation*}
By the First Main Theorem, the right-hand side of the above inequality is bounded from above by $3 \,T_f(r)+S_f(r)$. Thus we get
\[
5
\,
T_f(r)
\,
\leq
\,
3\,T_f(r)+S_f(r),
\]
which is absurd.

Next, we consider case \textbf{(a4)}. Assume that $\ord_z (h_1\circ f)\geq 3$ for all $z\in f^{-1}(A)$. Since $f$ is of the form \eqref{form of f_case1}, we claim that
\begin{equation}
\label{-ord-z-h2,h3-<=2}
\ord_z (h_i\circ f)
\,
\leq
\,
2\quad\quad {\scriptstyle(z\in \,f^{-1}(A),\quad\,2\,\leq\, i\,\leq\,3)}.
\end{equation}
Indeed, if $\ord_z(h_i\circ f)\geq 3$ for some $z\in f^{-1}(A)$ and for some $2\leq i\leq 3$, say $i=2$, then $(e^{\lambda_1\, z+\mu_1},e^{\lambda_2\, z+\mu_2},e^{\lambda_3\, z+\mu_3})$ is a solution of a system of six linear equations of the form
\[
\left\{
\aligned
0
&
\,=\,
a_{10}+a_{11}\,u+a_{12}\,v+a_{13}\,w,
\\
0
&
\,=\,
\ \ \ \ \ \ \ \
a_{11}\,\lambda_1\,u+a_{12}\,\lambda_2\,v+a_{13}\,\lambda_3\,w,
\\
0
&
\,=\,
\ \ \ \ \ \ \ \
a_{11}\,\lambda_1^2\,u+a_{12}\,\lambda_2^2\,v+a_{13}\,\lambda_3^2\,w,
\\
0
&
\,=\,
a_{20}+a_{21}\,u+a_{22}\,v+a_{23}\,w,
\\
0
&
\,=\,
\ \ \ \ \ \ \ \
a_{21}\,\lambda_1\,u+a_{22}\,\lambda_2\,v+a_{23}\,\lambda_3\,w,
\\
0
&
\,=\,
\ \ \ \ \ \ \ \
a_{21}\,\lambda_1^2\,u+a_{22}\,\lambda_2^2\,v+a_{23}\,\lambda_3^2\,w,\\
\endaligned\right.
\]
where $u$, $v$, $w$ are unknowns, and where $a_{ij}$ ($0\leq i\leq 3$) are the coefficients of $h_i$ ($1\leq i\leq 2$) in the homogeneous coordinate $[z_0:z_1:z_2:z_3]$. Since $\lambda_i$ are nonzero distinct constants, this forces the two linear forms $h_1$, $h_2$ to be linearly dependent, which is a contradiction.

It follows from \eqref{-ord-z-h2,h3-<=2} that
\begin{equation*}
\min
\,
\{\ord_z (h_2\circ f),3\}
+
\min
\,
\{\ord_z (h_3\circ f),3\}
\,
\leq 
\,
3
\min_{1\leq i\leq 3}
\,
\ord_z(h_i\circ f)\quad\quad {\scriptstyle(z\in \,f^{-1}(A))}.
\end{equation*}
By taking the sum on disks and by integrating, we get
\begin{equation}
\label{-estimate-counting-function-at-point-A-and-H2-H-3}
N_f^{[3]}(r,H_2)
+
N_f^{[3]}(r,H_3)
\,
\leq
\,
3\,N_f(r,A).
\end{equation}
We may therefore proceed similarly as in \eqref{-smt-m=2-case1.1}
\begin{align}\label{-using-smt-m=2-case-point+line-ord_z>=3}
5
\,
T_f(r)
&
\,
\leq
\,
\sum_{i=1}^{9}N_f^{[3]}(r,H_i)+S_f(r)\notag\\
&
\,
\leq
\,
N_f^{[3]}(r,H_1)+ 3\,N_f(r,A)+4\,N_f(r,l^*)+S_f(r)\notag\\
&
\,
\leq
\,
N_f(r,H_1)
+
4
\,
\big(
N_f(r,A)
+
N_f(r,l^*)
\big)
-
N_f(r,A)
+
S_f(r)\notag\\
&
\,
\leq 
\,
T_f(r)
+
4
\,
N_f(r,\mathcal{D})
-
N_f(r,A)
+
S_f(r)\notag\\
&
\,
\leq
\,
5
\,
T_f(r)
-
N_f(r,A)
+
S_f(r).
\end{align}
This implies
\[
N_f(r,A)
\,
=
\,
S_f(r).
\]
By \eqref{-estimate-counting-function-at-point-A-and-H2-H-3}, we have
\[
N_f^{[3]}(r,H_2)
+
N_f^{[3]}(r,H_3)
\,
=
\,
S_f(r).
\]
Hence we can rewrite the first inequality of \eqref{-using-smt-m=2-case-point+line-ord_z>=3} and use First Main Theorem to get a contradiction
\begin{align*}
5\,T_f(r)
&
\,
\leq
\,
N_f^{[3]}(r,H_1)
+
N_f^{[3]}(r,H_4)
+
N_f^{[3]}(r,H_5)
+
S_f(r)\\
&
\,
\leq
\,
3\,T_f(r)
+
S_f(r).
\end{align*}

Let us consider case \textbf{(b)}. Assume now $A_{2,3}=\{A,B\}$, where $A$, $B$ are two points contained in $\cup_{1\leq i_1<i_2<i_3\leq 9}\,H_{i_1}\cap H_{i_2}\cap H_{i_3}$. There are three possibilities for the positions of $A$ and $B$:
\begin{itemize}
\smallskip\item[\textbf{(b1)}] both $A$ and $B$ are contained in some line $H_i\cap H_j$;
\smallskip\item[\textbf{(b2)}] both $A$ and $B$ are contained in some plane $H_i$ but they are not contained in any line $H_i\cap H_j$;
\smallskip\item[\textbf{(b3)}] there is no plane $H_i$ containing both points $A$ and $B$.
\end{itemize}

In case \textbf{(b1)}, the curve $f$ avoids a family of five planes and, therefore, its image is contained in some diagonal plane of this family, which contains neither $A$ nor $B$ by the generic condition. Hence $f$ avoids all planes $H_i$, which is absurd by Theorem~\ref{fujimoto-green}.

Next, we consider case \textbf{(b2)}. Assume that $A=H_1\cap H_2\cap H_3$ and $B=H_1\cap H_4\cap H_5$, hence $f$ avoids  the $4$ planes $H_i$ ($6\leq i\leq 9$). Similarly as in case \textbf{(a)}, the generic condition allows us to assume that $f$ is linearly nondegenerate.

Since $f$ avoids four planes, it is of the form \eqref{form of f_case1} in some affine coordinates on $\mathbb{P}^3(\mathbb{C})$. Since $f$ has no singular point, we have
\begin{align}
\label{-order-equal-1}
\min_{i\in\{1,2,3\}}
\,
\ord_z (h_i\circ f)
&
\,
=
\,
1 \quad\quad {\scriptstyle(z\in \,f^{-1}(A))},\notag\\
\min_{i\in\{1,4,5\}}
\,
\ord_z (h_i\circ f)
&
\,
=
\,
1\quad\quad {\scriptstyle(z\in \,f^{-1}(B))}.
\end{align}
Hence by using these two equalities together 
with~\eqref{-ord-z-h2,h3-<=2},
\begin{equation*}
\sum_{i\in\{1,2,3\}}
\min\,\{\ord_z(h_i\circ f),3\}
\,
\leq
\,
6
\,
=
\,
6
\,
\min_{i\in\{1,2,3\}}
\,
\ord_z(h_i\circ f),\quad\quad {\scriptstyle(z\,\in\,f^{-1}(A))},
\end{equation*}
\begin{equation*}
\sum_{i\in\{1,4,5\}}
\min\,\{\ord_z(h_i\circ f),3\}
\,
\leq
\,
6
\,
=
\,
6
\,
\min_{i\in\{1,4,5\}}
\,
\ord_z(h_i\circ f),\quad\quad {\scriptstyle(z\,\in\,f^{-1}(B))}.
\end{equation*}
Thus, by taking the sum on disks of both sides of these inequalities and by integrating,
\begin{equation*}
\sum_{i=1}^5
N_f^{[3]}(r,H_i)
\,
\leq
\,
6
\,
\big(
N_f(r,A)
+
N_f(r,B)
\big).
\end{equation*}

Next, using again that $f$ is of the form \eqref{form of f_case1}, one can find two planes $\mathcal{P}_1=\{\mathsf{p}_1=0 \}$, $\mathcal{P}_2=\{\mathsf{p}_2=0 \}$ containing the line $AB$ such that
\[
\ord_z(\mathsf{p}_1\circ f)
\,
\geq
\,
2
\quad\quad {\scriptstyle(z\,\in\,f^{-1}(A))},
\]
\[
\ord_z(\mathsf{p}_2\circ f)
\,
\geq
\,
2
\quad\quad {\scriptstyle(z\,\in\,f^{-1}(B))}.
\]
Let $\mathcal{Q}=\{\mathsf{q} = \mathsf{p}_1\mathsf{p}_2=0\}$ be the degenerate quadric $\mathcal{P}_1\cup\mathcal{P}_2$. We have
\[
3
\,
=
\,
3
\,
\min_{i\in\{1,2,3\}}
\,
\ord_z(h_i\circ f)
\,
\leq
\,
\ord_z(\mathsf{p}_1\circ f)+\ord_z(\mathsf{p}_2\circ f)
\,
=
\,
\ord_z(\mathsf{q}\circ f) \quad\quad {\scriptstyle(z\,\in\,f^{-1}(A))},
\]
\[
3
\,
=
\,
3
\,
\min_{i\in\{1,4,5\}}\,
\ord_z(h_i\circ f)
\,
\leq
\,
\ord_z(\mathsf{p}_1\circ f)
+
\ord_z(\mathsf{p}_2\circ f)
\,
=
\,
\ord_z(\mathsf{q}\circ f) \quad\quad {\scriptstyle(z\,\in\,f^{-1}(B))},
\]
which implies, by integrating, that
\begin{equation*}
3
\,
\big(
N_f(r,A)
+
N_f(r,B)
\big)
\,
\leq
\,
N_f(r,\mathcal{T}).
\end{equation*}
We proceed similarly as above to get a contradiction
\begin{align*}
5
\,
T_f(r)
&
\,
\leq
\,
\sum_{i=1}^{9}
N_f^{[3]}(r,H_i)
+
S_f(r)\\
&
\,
\leq
\,
6
\,
\big(
N_f(r,A)
+
N_f(r,B)
\big)
+S_f(r)\\
&
\,
\leq
\,
2
\,
N_f(r,\mathcal{T})
+
S_f(r)\\
&
\,
\leq
\,
4
\,
T_f(r)
+
S_f(r).
\end{align*}

Now, we consider case \textbf{(b3)}. Assume that $A=H_1\cap H_2\cap H_3$, $B=H_4\cap H_5\cap H_6$, when $f$ avoids the three planes $H_7$, $H_8$, $H_9$. If $f(\mathbb{C})$ is contained in some plane $\mathcal{P}$, then it is not hard to see that $\mathcal{P}$ must pass through both $A$ and $B$. Furthermore, by using Theorem~\ref{-hyperplane-not-in-general-position}, one can show that $\mathcal{P}$ does not pass through the point $C=H_7\cap H_8\cap H_9$. One can then always find $7$ lines in general position in $\mathcal{P}$ among $\{H_i\cap\mathcal{P}\}_{1\leq i\leq 9}$. Hence one can use similar arguments as in Lemma~\ref{start_lem_P2}, case $m=2$, to get a contradiction. Thus, we can suppose that $f$ is linearly nondegenerate.

Assume that the omitted planes $H_7$, $H_8$, $H_9$ are given
in the homogeneous coordinates $[z_0:z_1:z_2:z_3]$ by the equations $\{z_0=0\}$, $\{z_1=0\}$, $\{z_2=0\}$. Since $\{H_i\}_{1\leq i\leq 9}$ is a family of planes in general position, the planes $H_i$ ($1\leq i\leq 6$) are given by
\[
h_i
\,
=
\,
\sum_{j=0}^3 a_{ij}\, z_j
\,
=
\,
0,
\]
with $a_{i3}\not=0$ ($1\leq i\leq 6$). Set $l_{i_1,i_2}=H_{i_1}\cap H_{i_2}$ ($1\leq i_1<i_2\leq 3$), $l_{j_1,j_2}=H_{j_1}\cap H_{j_2}$ ($4\leq j_1< j_2\leq 6$). For $1\leq i<j\leq3$ or $4\leq i<j\leq 6$, let $R_{i,j}=\{r_{i,j}=0\}$ be the plane containing the lines $AB$, $l_{i,j}$ and let $T_{i,j}=\{t_{i,j}=a_{j3}\,h_i-a_{i3}\,h_j=0\}$ be the plane passing through the point $C=[0:0:0:1]$ and containing the line $l_{i,j}$. We note that all $r_{i,j}$, $t_{i,j}$ are linear combinations of $h_i$ and $h_j$ with nonzero coefficients.

Since $f$ avoids three planes, by Theorem~\ref{omit-hyperplanes}  it has a reduced representation of the form
\begin{equation}
\label{-form-of-f-case-2}
[1:e^{\lambda_1\,z+\mu_1}:e^{\lambda_2\,z+\mu_2}:g],
\end{equation}
where $\lambda_1$, $\lambda_2$, $\mu_1$, $\mu_2$ are constants with $\lambda_1\not=\lambda_2$, $\lambda_1,\lambda_2\not=0$ and where $g$ is an entire function. Since $f$ has no singular point, we have
\begin{align}
\label{-order-equal-1}
\min_{1\leq i\leq 3}
\,
\ord_z (h_i\circ f)
&
\,
=
\,
1
\quad\quad {\scriptstyle(z\in \,f^{-1}(A))},\notag\\
\min_{4\leq j\leq 6}
\,
\ord_z (h_j\circ f)
&
\,
=
\,
1
\quad\quad {\scriptstyle(z\in \,f^{-1}(B))}.
\end{align}
Since $f$  is of the form \eqref{-form-of-f-case-2}, we claim that
\begin{align}
\label{-min-order-hi-hj-less-than-2}
\min
\,
\{
\ord_z(h_{i_1}\circ f)
,
\ord_z(h_{i_2}\circ f)
\}
&
\,
\leq
\,
2
\quad\quad {\scriptstyle(z\in \,f^{-1}(A)\quad\,1\,\leq\, i_1\,<\,i_2\,\leq\,3)},\notag\\
\min
\,
\{
\ord_z(h_{j_1}\circ f)
,
\ord_z(h_{j_2}\circ f)
\}
&
\,
\leq
\,
2
\quad\quad {\scriptstyle(z\in \,f^{-1}(B),\quad\,4\,\leq\, j_1\,<\,j_2\,\leq\,6)}.
\end{align}
Indeed, if one of these inequalities does not hold, say $\min\,\{\ord_z(h_1\circ f),\ord_z(h_2\circ f)\}\geq 3$ for some $z\in f^{-1}(A)$, then $z$ is a solution of the following system of equations
\[
\left\{
\aligned
0
&
\,=\,
(
t_{1,2}
\circ
f
)
(z)
,
\\
0
&
\,=\,
(
t_{1,2}
\circ
f
)'
(z),
\\
0
&
\,=\,
(
t_{1,2}
\circ
f
)''
(z).
\endaligned\right.
\]
Equivalently, $(e^{\lambda_1\,z+\mu_1},e^{\lambda_2\, z+\mu_2})$ is a solution of a system of three linear equations of the form
\[
\left\{
\aligned
0
&
\,=\,
(
a_{23}\,a_{10}
-a_{13}\,a_{20}
)
+
(
a_{23}\,a_{11}
-a_{13}\,a_{21}
)
\,
x
\ \ \ \
+
(
a_{23}\,a_{12}
-a_{13}\,a_{22}
)
\,
y,
\\
0
&
\,=\,
\ \ \ \ \ \ \ \ \ \ \ \ \ \ \ \ \ \ \ \ \ \ \ \ \ \ \
(
a_{23}\,a_{11}
-a_{13}\,a_{21}
)
\,
\lambda_1
\,
x
+
(
a_{23}\,a_{12}
-a_{13}\,a_{22}
)
\,
\lambda_2
\,
y,
\\
0
&
\,=\,
\ \ \ \ \ \ \ \ \ \ \ \ \ \ \ \ \ \ \ \ \ \ \ \ \ \ \
(
a_{23}\,a_{11}
-a_{13}\,a_{21}
)
\,
\lambda_1^2
\,
x
+
(
a_{23}\,a_{12}
-a_{13}\,a_{22}
)
\,
\lambda_2^2
\,
y,
\\
\endaligned\right.
\]
where $x$, $y$ are unknowns. Since $\lambda_1\not=\lambda_2$, $\lambda_1, \lambda_2\not=0$, this implies that the two linear forms $h_1$, $h_2$ must be linearly dependent, which is a contradiction.

It follows from \eqref{-order-equal-1} and \eqref{-min-order-hi-hj-less-than-2} that
\begin{align}
\label{-first-estimate-ord_z-h_i}
\sum_{i=1}^3
\min
\,
\{
\ord_z (h_{i}\circ f),
3
\}
&
\,
\leq 
\,
6
\quad\quad 
{\scriptstyle(z\,\in \,f^{-1}(A))},
\notag
\\
\sum_{j=4}^6
\min
\{
\ord_z (h_j\circ f),3
\}
&
\,
\leq
\,
6 \quad\quad {\scriptstyle(z\,\in \,f^{-1}(B))}.
\end{align}
Now we prove the following equality
\begin{claim}
\label{-claim}
\begin{equation}
\label{-equation-counting-A,B-order-case-b31}
T_f(r)
\,
=
\,
N_f(r,A)
+
N_f(r,B)
+
S_f(r).
\end{equation}
\end{claim}
\begin{proof}
Since $f$ is of the form \eqref{-form-of-f-case-2} and since $t_{i,j}$ does not contain the term $x_3$, we have
\begin{align}
\label{-ord-does-not-change-case2}
\ord_{z_1}(t_{i_1,i_2}\circ f)
&
\,
=
\,
\ord_{z_2}(t_{i_1,i_2}\circ f)
\quad\quad 
{\scriptstyle (z_1,z_2\,\in\,f^{-1}(A),\,1\,\leq\,i_1\,<i_2\,\,\leq\,3)},\notag\\
\ord_{z_1}(t_{j_1,j_2}\circ f)
&
\,
=
\,
\ord_{z_2}(t_{j_1,j_2}\circ f)
\quad\quad 
{\scriptstyle (z_1,z_2\,\in\,f^{-1}(B),\,4\,\leq\,j_1\,<j_2\,\,\leq\,6)}.
\end{align}
Thus, it suffices to consider the four cases depending on $f$ and $t_{i,j}$:
\begin{itemize}
\smallskip\item[\textbf{(b3.1)}] $\ord_z(t_{i_1,i_2}\circ f)=1$ for all $1\leq i_1<i_2\leq 3$, for all $z\in f^{-1}(A)$ and $\ord_z(t_{j_1,j_2}\circ f)=1$ for all $4\leq j_1<j_2\leq 6$, for all $z\in f^{-1}(B)$;
\smallskip\item[\textbf{(b3.2)}] $\ord_z(t_{i_1,i_2}\circ f)\geq 2$ for some $1\leq i_1<i_2\leq 3$, for all $z\in f^{-1}(A)$ and $\ord_z(t_{j_1,j_2}\circ f)=1$ for all $4\leq j_1<j_2\leq 6$, for all $z\in f^{-1}(B)$;
\smallskip\item[\textbf{(b3.3)}] $\ord_z(t_{i_1,i_2}\circ f)=1$ for all $1\leq i_1<i_2\leq 3$, for all $z\in f^{-1}(A)$ and $\ord_z(t_{j_1,j_2}\circ f)\geq 2$ for some $4\leq j_1<j_2\leq 6$, for all $z\in f^{-1}(B)$;
\smallskip\item[\textbf{(b3.4)}] $\ord_z(t_{i_1,i_2}\circ f)\geq 2$ for some $1\leq i_1<i_2\leq 3$, for all $z\in f^{-1}(A)$ and $\ord_z(t_{j_1,j_2}\circ f)\geq 2$ for some $4\leq j_1<j_2\leq 6$, for all $z\in f^{-1}(B)$.
\end{itemize}

Consider case \textbf{(b3.1)}. Since $t_{i,j}$ is a linear combination of $h_i$ and $h_j$ with nonzero coefficients, we have
\begin{align*}
\min
\,
\{
\ord_z(h_{i_1}\circ f)
,
\ord_z(h_{i_2}\circ f)
\}
&
\,
=
\,1
\quad\quad {\scriptstyle(z\in \,f^{-1}(A)\quad\,1\,\leq\, i_1\,<\,i_2\,\leq\,3)},\\
\min
\,
\{
\ord_z(h_{j_1}\circ f)
,
\ord_z(h_{j_2}\circ f)
\}
&
\,
=
\,
1
\quad\quad {\scriptstyle(z\in \,f^{-1}(B),\quad\,4\,\leq\, j_1\,<\,j_2\,\leq\,6)}.
\end{align*}
Using these equalities together with \eqref{-order-equal-1}, we get
\begin{align} 
\label{-ord-totall<=5}
\sum_{i=1}^3
\min
\,
\{
\ord_z (h_{i}\circ f)
,
3
\}
&
\,
\leq
\,
5
\,
=
\,
5
\,
\min_{1\leq i\leq 3}
\,
\ord_z (h_i\circ f)\quad\quad {\scriptstyle(z\in \,f^{-1}(A))},\notag\\
\sum_{i=4}^6
\min
\,
\{\ord_z (h_i\circ f),3\}
&
\,
\leq
\,
5
\,
=
\,
5
\,
\min_{4\leq i\leq 6}
\,
\ord_z (h_i\circ f) \quad\quad {\scriptstyle(z\in \,f^{-1}(B))}.
\end{align}
By taking the sum on disks and by integrating these two inequalities, we obtain
\begin{equation*}
\label{-compare-counting-function-case-b31-h1,h2,h3-A}
N_f^{[3]}(r,H_1)
+
N_f^{[3]}(r,H_2)
+
N_f^{[3]}(r,H_3)
\,
\leq
\, 
5
\,
N_f(r,A),
\end{equation*}
\begin{equation*}
\label{-compare-counting-function-case-b31-h4,h5,h6-B}
N_f^{[3]}(r,H_4)
+
N_f^{[3]}(r,H_5)
+
N_f^{[3]}(r,H_6)
\,
\leq 
\,
5
\,
N_f(r,B).
\end{equation*}
Letting $\mathcal{B}$ be a plane passing through $A$ and $B$, we proceed similarly as before
\begin{align}\label{-estimate-order-counting-case-b31}
5
\,
T_f(r)
&
\,
\leq
\,
\sum_{i=1}^{9}
N_f^{[3]}(r,H_i)
+
S_f(r)\notag\\
&
\,
\leq
\,
5
\,
N_f(r,A)
+
5
\,
N_f(r,B)
+
S_f(r)\notag\\
&
\,
\leq
\,
5
\,
N_f(r,\mathcal{B})
+
S_f(r)\notag\\
&
\,
\leq
\,
5
\,
T_f(r)
+
S_f(r).
\end{align}
Here, $S_f(r)=o(T_f(r))$ is negligible, hence all inequalities are equalities modulo $S_f(r)$. This gives \eqref{-equation-counting-A,B-order-case-b31}, as wanted.

Next, we consider case \textbf{(b3.2)}. Let us set
\begin{align*}
E_{t,A}
&
\,
=
\,
\{z\in\mathbb{C}\colon|z|<t,f(z)=A\},\\
E_{t,A,i}^1
&
\,
=
\,
\{z\in\mathbb{C}\colon|z|<t,f(z)=A,\ord_z (h_i\circ f)=1\}\quad\quad {\scriptstyle(1\,\leq\,i\,\leq\,3)},\\
E_{t,A,i}^{\geq 2}
&
\,
=
\,
\{z\in\mathbb{C}\colon|z|<t,f(z)=A,\ord_z (h_i\circ f)\geq2\}\quad\quad {\scriptstyle(1\,\leq\,i\,\leq\,3)},\\
E_{t,B}
&
\,
=
\,
\{z\in\mathbb{C}\colon|z|<t,f(z)=B\},\\
E_{t,B,i}^1
&
\,
=
\,
\{z\in\mathbb{C}\colon|z|<t,f(z)=B,\ord_z (h_i\circ f)=1\}\quad\quad {\scriptstyle(4\,\leq\,i\,\leq\,6)},\\
E_{t,B,i}^{\geq 2}
&
\,
=
\,
\{z\in\mathbb{C}\colon|z|<t,f(z)=B,\ord_z (h_i\circ f)\geq2\}\quad\quad {\scriptstyle(4\,\leq\,i\,\leq\,6)}.
\end{align*}
Assume that $\ord_z(t_{1,2}\circ f)\geq 2$ for all $z\in f^{-1}(A)$. Since $t_{1,2}$, $r_{1,2}$ are linear combinations of $h_1$ and $h_2$ with nonzero coefficients, we have
\begin{align}
E_{t,A,1}^{\geq 2}
&
\,
=
\,
E_{t,A,2}^{\geq 2},\notag\\
\label{-case3-ord-r-1}
\ord_z(r_{1,2}\circ f)
&
\,
\geq
\,
2
\quad\quad\quad\quad {\scriptstyle(z\,\in\, E_{t,A,1}^{\geq 2})}.
\end{align}
For the same reason
\[
E_{t,A,1}^1
\,
=
\,
E_{t,A,2}^1,
\]
which yields
\begin{equation}
\label{-case3-ord-h1,h2,h3,min,3,leq,5}
\sum_{i=1}^3
\min
\,
\{\ord_z (h_i\circ f),3\}
\,
\leq
\,
5
\quad\quad {\scriptstyle(z\,\in\, E_{t,A,1}^1)}.
\end{equation}
Letting $\mathcal{R}=\{\mathsf{r}=r_{1,2}\,r_{4,5}\,r_{5,6}\,r_{4,6}=0\}$ be the degenerate quartic $R_{1,2}\cup R_{4,5}\cup R_{5,6}\cup R_{4,6}$ whose four components pass through $A$ and $B$, we have
\begin{equation}
\label{-ord-A,B-always-greater-than-4}
\ord_z(\mathsf{r}\circ f)
\,
\geq
\,
4\quad\quad {\scriptstyle(z\in E_{t,A}\cup E_{t,B})}.
\end{equation}
Furthermore, it follows from~\eqref{-case3-ord-r-1} that
\[
\ord_z(\mathsf{r}\circ f)
\,
\geq
\,
5
\quad\quad {\scriptstyle(z\,\in\, E_{t,A,1}^{\geq 2})}.
\]
Using this inequality together with \eqref{-first-estimate-ord_z-h_i} and \eqref{-order-equal-1}, we get
\begin{equation*}
\sum_{i=1}^3
\min
\,
\{\ord_z (h_i\circ f),3\}
\,
\leq
\,
6
\,
=
\,
6
\,
\min_{1\leq i\leq 3}
\,
\ord_z(h_i\circ f)
\,
\leq
\,
\dfrac{6}{5}
\ord_z(\mathsf{r}\circ f)
\quad\quad {\scriptstyle(z\,\in\, E_{t,A,1}^{\geq 2})}.
\end{equation*}
Combining \eqref{-case3-ord-h1,h2,h3,min,3,leq,5}, \eqref{-order-equal-1} and \eqref{-ord-A,B-always-greater-than-4}, we receive
\begin{equation*}
\sum_{i=1}^3
\min
\,
\{\ord_z (h_i\circ f),3\}
\,
\leq
\,
5
\,
=
\,
5\min_{1\leq i\leq 3}
\,
\ord_z (h_i\circ f)
\,
\leq
\,
\dfrac{5}{4}\ord_z(\mathsf{r}\circ f)
\quad\quad {\scriptstyle(z\,\in\, E_{t,A,1}^1)}.
\end{equation*}
Since $\ord_z(t_{j_1,j_2}\circ f)=1$ for all $4\leq j_1<j_2\leq 6$, for all $z\in f^{-1}(B)$, by similar arguments as in \eqref{-ord-totall<=5} and by using \eqref{-ord-A,B-always-greater-than-4}, we also have
\begin{equation*}
\sum_{i=4}^6
\min
\,
\{\ord_z (h_i\circ f),3\}
\,
\leq
\,
5
\,
=
\,
5
\,
\min_{4\leq i\leq 6}
\,
\ord_z(h_i\circ f)
\,
\leq
\,
\dfrac{5}{4}
\ord_z(\mathsf{r}\circ f)
\quad\quad {\scriptstyle(z\in \,E_{t,B})}.
\end{equation*}
By taking the sum on disks and by integrating these three inequalities, we obtain
\begin{align}
\label{-estimate-ord-h1,h2,h3-r1,r4,r5,r6-Et,a,i}
\sum_{i=1}^3
\int_1^r
\dfrac{\sum_{z\in E_{t,A,1}^{\geq2}}
\min
\,
\{\ord_z (h_i\circ f),3\}}{t}\,\der t
&
\,
\leq
\,
6
\,
\int_1^r
\dfrac{\sum_{z\in E_{t,A,1}^{\geq2}}
\min_{1\leq i\leq 3}
\,
\ord_z (h_i\circ f)}{t}\,\der t\notag\\
&
\,
\leq
\,
\dfrac{6}{5}
\int_1^r
\dfrac{\sum_{z\in E_{t,A,1}^{\geq 2}}\ord_z (\mathsf{r}\circ f)}{t}\,\der t,\\
\label{-estimate-ord-h1,h2,h3-r1,r4,r5,r6-E't,a,i}
\sum_{i=1}^3
\int_1^r
\dfrac{\sum_{z\in E_{t,A,1}^1}
\min
\,
\{\ord_z (h_i\circ f),3\}}{t}\,\der t
&
\,
\leq
\,
5
\int_1^r
\dfrac{\sum_{z\in E_{t,A,1}^1}\min_{1\leq i\leq 3}
\,
\ord_z (h_i\circ f)}{t}\,\der t\notag\\
&
\,
\leq
\,
\dfrac{5}{4}
\int_1^r
\dfrac{\sum_{z\in E_{t,A,1}^1}\ord_z (\mathsf{r}\circ f)}{t}\,\der t,\\
\label{-estimate-ord-h4,h5,h6-r1,r4,r5,r6-E't,b,i}
\sum_{i=4}^6
\int_1^r
\dfrac{\sum_{z\in E_{t,B}}
\min
\,
\{\ord_z (h_i\circ f),3\}}{t}\,\der t
&
\,
\leq
\,
5\, N_f(r,B)\notag\\
&
\,
\leq
\,
\dfrac{5}{4}
\int_1^r
\dfrac{\sum_{z\in E_{t,B}}\ord_z (\mathsf{r}\circ f)}{t}\,\der t.
\end{align}
We then proceed similarly as before:
\begin{align}\label{-smt-tangent-1-line}
5
\,
T_f(r)
&
\,
\leq
\,
\sum_{i=1}^9
N_f^{[3]}(r,H_i)
+
S_f(r)\notag\\
&
\,
\leq
\,
\sum_{i=1}^3
N_f^{[3]}(r,H_i)
+
\sum_{i=4}^6
N_f^{[3]}(r,H_i)
+
S_f(r)\notag\\
&
\,
=
\,
\sum_{i=1}^3
\int_1^r
\dfrac{\sum_{z\in E_{t,A}}
\min
\,
\{\ord_z (h_i\circ f),3\}}{t}\,\der t
+
\sum_{i=4}^6
\int_1^r
\dfrac{\sum_{z\in E_{t,B}}
\min
\,
\{\ord_z (h_i\circ f),3\}}{t}\,\der t
+S_f(r)\notag\\
&
\,
=
\,
\sum_{i=1}^3
\bigg(
\int_1^r
\dfrac{\sum_{z\in E_{t,A,1}^{\geq2}}
\min
\,
\{\ord_z (h_i\circ f),3\}}{t}\,\der t
+
\int_1^r
\dfrac{\sum_{z\in E_{t,A,1}^1}
\min
\,
\{\ord_z (h_i\circ f),3\}}{t}\,\der t
\bigg)
\notag\\
&
\,
+
\,
\sum_{i=4}^6
\int_1^r
\dfrac{\sum_{z\in E_{t,B}}
\min
\,
\{\ord_z (h_i\circ f),3\}}{t}\,\der t
+
S_f(r)\notag\\
&
\,
\leq
\,
\dfrac{6}{5}
\int_1^r
\dfrac{\sum_{z\in E_{t,A,1}^{\geq2}}\ord_z (\mathsf{r}\circ f)}{t}\,\der t
+\dfrac{5}{4}
\int_1^r
\dfrac{\sum_{z\in E_{t,A,1}^1}\ord_z (\mathsf{r}\circ f)}{t}\,\der t\notag\\
&
\,
+
\,
\dfrac{5}{4}\int_1^r\dfrac{\sum_{z\in E_{t,B}}\ord_z (\mathsf{r}\circ f)}{t}\,\der t
+
S_f(r)
\notag\\
&
\,
=
\,
\dfrac{5}{4}
\,
\bigg(
\int_1^r
\dfrac{\sum_{z\in E_{t,A}}\ord_z (\mathsf{r}\circ f)}{t}\,\der t
+
\int_1^r
\dfrac{\sum_{z\in E_{t,B}}\ord_z (\mathsf{r}\circ f)}{t}\,\der t
\bigg)
\notag\\
&
\,
+
\,
\bigg(
\dfrac{6}{5}
-
\dfrac{5}{4}
\bigg)
\int_1^r
\dfrac{\sum_{z\in E_{t,A,1}^{\geq2}}\ord_z (\mathsf{r}\circ f)}{t}\,\der t
+
S_f(r)\notag\\
&
\,
\leq
\,
\dfrac{5}{4}
\,
N_f(r,\mathcal{R})
-
\dfrac{1}{20}
\int_1^r
\dfrac{\sum_{z\in E_{t,A,1}^{\geq2}}\ord_z (\mathsf{r}\circ f)}{t}\,\der t
+
S_f(r)\notag\\
&
\,
\leq
\,
5
\,
T_f(r)
-
\dfrac{1}{20}
\int_1^r
\dfrac{\sum_{z\in E_{t,A,1}^{\geq2}}\ord_z(\mathsf{r}\circ f)}{t}\,\der t
+
S_f(r).
\end{align}
This implies
\begin{equation}
\label{-small-equation-b32}
\int_1^r
\dfrac{\sum_{z\in E_{t,A,1}^{\geq2}}\ord_z (\mathsf{r}\circ f)}{t}\,\der t
\,
=
\,
S_f(r)
\end{equation}
and whence all inequalities in \eqref{-smt-tangent-1-line} become equalities modulo $S_f(r)$, which gives
\begin{align}
\label{-first-equation-b32}
\sum_{i=1}^6
N_f^{[3]}(r,H_i)
&
\,
=
\,
5
\,
T_f(r)
+
S_f(r),\\
\label{-second-equation-b32}
\sum_{i=1}^3
N_f^{[3]}(r,H_i)
&
\,
=
\,
\dfrac{5}{4}
\int_1^r
\dfrac{\sum_{z\in E_{t,A}}\ord_z (\mathsf{r}\circ f)}{t}\,\der t
+
S_f(r),\\
\label{-third-equation-b32}
\sum_{i=4}^6
N_f^{[3]}(r,H_i)
&
\,
=
\,
\dfrac{5}{4}
\int_1^r
\dfrac{\sum_{z\in E_{t,B}}\ord_z (\mathsf{r}\circ f)}{t}\,\der t
+
S_f(r).
\end{align}
It follows from \eqref{-estimate-ord-h4,h5,h6-r1,r4,r5,r6-E't,b,i} and \eqref{-third-equation-b32} that
\begin{equation}
\label{-equation-h4h5h6,B}
\sum_{i=4}^6
N_f^{[3]}(r,H_i)
\,
=
\,
5
\,
N_f(r,B)
+
S_f(r).
\end{equation}
Owing to \eqref{-small-equation-b32}, the two inequalities \eqref{-estimate-ord-h1,h2,h3-r1,r4,r5,r6-Et,a,i} become
\[
\sum_{i=1}^3
\int_1^r
\dfrac{\sum_{z\in E_{t,A,1}^{\geq2}}
\min
\,
\{\ord_z (h_i\circ f),3\}}{t}\,\der t
\,
=
\,
S_f(r)
\]
\[
\int_1^r
\dfrac{\sum_{z\in E_{t,A,1}^{\geq2}}
\min_{1\leq i\leq 3}
\,
\ord_z (h_i\circ f)}{t}\,\der t
\,
=
\,
S_f(r).
\]
Hence
\begin{equation}
\label{-equation-h1h2h3-again}
\sum_{i=1}^3
N_f^{[3]}(r,H_i)
\,
=
\,
\sum_{i=1}^3
\int_1^r
\dfrac{\sum_{z\in E_{t,A,1}^1}
\min\,
\{\ord_z (h_i\circ f),3\}}{t}\,\der t
+
S_f(r),
\end{equation}
\begin{equation}
\label{-equation-A-EtAi-again}
N_f(r,A)
\,
=
\,
\int_1^r
\dfrac{\sum_{z\in E_{t,A,1}^1}
\min_{1\leq i\leq 3}\,
\ord_z (h_i\circ f)}{t}\,\der t
+
S_f(r).
\end{equation}
Combining~\eqref{-estimate-ord-h1,h2,h3-r1,r4,r5,r6-E't,a,i}, \eqref{-equation-h1h2h3-again}, \eqref{-equation-A-EtAi-again}, we get
\begin{equation}
\label{-equation-h1h2h3,A}
\sum_{i=1}^3
N_f^{[3]}(r,H_i)
\,
=
\,
5
\,
N_f(r,A)
+
S_f(r).
\end{equation}
The equality \eqref{-equation-counting-A,B-order-case-b31} follows from \eqref{-first-equation-b32}, \eqref{-equation-h4h5h6,B}, \eqref{-equation-h1h2h3,A}.

\smallskip 

Case \textbf{(b3.3)} can be treated by similar arguments as for case \textbf{(b3.2)}.

\smallskip

Next, we consider case \textbf{(b3.4)}. Assume that
\[
\ord_z(t_{1,2}\circ f)
\,
\geq
\,
2 \quad\quad{\scriptstyle(z\,\in\,f^{-1}(A))},
\]
\[
\ord_z(t_{4,5}\circ f)
\,
\geq
\,
2 \quad\quad{\scriptstyle(z\,\in\,f^{-1}(B))}.
\]
By similar argument as in \eqref{-case3-ord-r-1}, we have $E_{t,A,1}^{\geq2}
=E_{t,A,2}^{\geq2}$, $E_{t,B,4}^{\geq2}=E_{t,B,5}^{\geq2}$, $E_{t,A,1}^1=E_{t,A,2}^1$, $E_{t,B,4}^1=E_{t,B,5}^1$, which implies
\begin{equation}
\label{-ord_r12_geq-2}
\ord_z(r_{1,2}\circ f)
\,
\geq
\,
2
\quad\quad {\scriptstyle(z\,\in\, E_{t,A,1}^{\geq2})},
\end{equation}
\begin{equation}
\label{-ord_r45_geq-2}
\ord_z(r_{4,5}\circ f)
\,
\geq
\,
2
\quad\quad {\scriptstyle(z\,\in\, E_{t,B,4}^{\geq2})},
\end{equation}
\[
\sum_{i=1}^3
\min
\,
\{\ord_z (h_i\circ f),3\}
\,
\leq
\,
5\quad\quad {\scriptstyle(z\,\in\, E_{t,A,1}^1)},
\]
\[
\sum_{i=4}^6
\min
\,
\{\ord_z (h_i\circ f),3\}
\,
\leq
\,
5\quad\quad {\scriptstyle(z\,\in\, E_{t,B,4}^1)}.
\]
Letting $\mathcal{S}=\{\mathsf{s}=r_{12}\,r_{4,5}=0\}$ be the degenerate quadric $R_{1,2}\cup R_{4,5}$, we see that
\[
\ord_z(\mathsf{s}\circ f)
\,
=
\,
\ord_z (r_{1,2}\circ f)
+
\ord_z (r_{4,5}\circ f)
\,
\geq
\,
2
\quad\quad {\scriptstyle(z\,\in\,E_{t,A}\,\cup\,E_{t,B})}.
\]
Furthermore, by using \eqref{-ord_r12_geq-2} and \eqref{-ord_r45_geq-2}, we have
\[
\ord_z(\mathsf{s}\circ f)
\,
=
\,
\ord_z (r_{1,2}\circ f)
+
\ord_z (r_{4,5}\circ f)
\,
\geq
\,
3
\quad\quad {\scriptstyle(z\,\in\, E_{t,A,1}^{\geq 2}\,\cup\,E_{t,B,4}^{\geq 2})}.
\]
Similarly as in the previous case, by using these inequalities together with \eqref{-order-equal-1} and \eqref{-first-estimate-ord_z-h_i}, we receive
\begin{equation*}
\sum_{i=1}^3
\min
\,
\{\ord_z (h_i\circ f),3\}
\,
\leq
\,
6
\,
=
\,
6
\,
\min_{1\leq i\leq 3}
\,
\ord_z(h_i\circ f)
\,
\leq
\,
\dfrac{6}{3}
\ord_z(\mathsf{s}\circ f)
\quad\quad {\scriptstyle(z\,\in\, E_{t,A,1}^{\geq 2})},
\end{equation*}
\begin{equation*}
\sum_{i=1}^3
\min
\,
\{\ord_z (h_i\circ f),3\}
\,
\leq
\,
5
\,
=
\,
5
\,
\min_{1\leq i\leq 3}
\,
\ord_z(h_i\circ f)
\,
\leq
\,
\dfrac{5}{2}
\ord_z(\mathsf{s}\circ f)
\quad\quad {\scriptstyle(z\,\in\, E_{t,A,1}^1)},
\end{equation*}
\begin{equation*}
\sum_{i=4}^6
\min
\,
\{\ord_z (h_i\circ f),3\}
\,
\leq
\,
6
\,
=
\,
6
\,
\min_{4\leq i\leq 6}
\,
\ord_z(h_i\circ f)
\,
\leq
\,
\dfrac{6}{3}
\ord_z(\mathsf{s}\circ f)
\quad\quad {\scriptstyle(z\,\in\, E_{t,B,4}^{\geq 2})},
\end{equation*}
\begin{equation*}
\sum_{i=4}^6
\min
\,
\{\ord_z (h_i\circ f),3\}
\,
\leq
\,
5
\,
=
\,
5
\,
\min_{4\leq i\leq 6}
\,
\ord_z(h_i\circ f)
\,
\leq
\,
\dfrac{5}{2}
\ord_z(\mathsf{s}\circ f)
\quad\quad {\scriptstyle(z\,\in\, E_{t,B,4}^1)}.
\end{equation*}
By taking the sum on disks and by integrating these four inequalities, we obtain
\begin{align*}
\label{-estimate-ord-h1,h2,h3-r1,r4-Et,a,i}
\sum_{i=1}^3
\int_1^r
\dfrac{\sum_{z\in E_{t,A,1}^{\geq 2}}
\min
\,
\{\ord_z (h_i\circ f),3\}}{t}\,\der t
&
\,
\leq
\,
6
\int_1^r
\dfrac{\sum_{z\in E_{t,A,1}^{\geq 2}}
\min_{1\leq i\leq 3}
\,
\ord_z (h_i\circ f)}{t}\,\der t\\
&
\,
\leq
\,
2
\int_1^r
\dfrac{\sum_{z\in E_{t,A,1}^{\geq 2}}
\ord_z (\mathsf{s}\circ f)}{t}\,\der t,
\end{align*}
\begin{align*}
\sum_{i=1}^3
\int_1^r
\dfrac{\sum_{z\in E_{t,A,1}^1}
\min
\,
\{\ord_z (h_i\circ f),3\}}{t}\,\der t
&
\,
\leq
\,
5
\int_1^r
\dfrac{\sum_{z\in E_{t,A,1}^1}
\min_{1\leq i\leq 3}
\,
\ord_z (h_i\circ f)}{t}\,\der t\\
&
\,
\leq
\,
\dfrac{5}{2}
\int_1^r
\dfrac{\sum_{z\in E_{t,A,1}^1}\ord_z (\mathsf{s}\circ f)}{t}\,\der t,
\end{align*}
\begin{align*}
\sum_{i=4}^6
\int_1^r
\dfrac{\sum_{z\in E_{t,B,4}^{\geq 2}}
\min
\,
\{\ord_z (h_i\circ f),3\}}{t}\,\der t
&
\,
\leq
\,
6
\int_1^r
\dfrac{\sum_{z\in E_{t,B,4}^{\geq 2}}
\min_{4\leq i\leq 6}
\,
\ord_z (h_i\circ f)}{t}\,\der t\\
&
\,
\leq
\,
2
\int_1^r
\dfrac{\sum_{z\in E_{t,B,4}^{\geq 2}}\ord_z (\mathsf{s}\circ f)}{t}\,\der t,
\end{align*}
\begin{align*}
\sum_{i=4}^6
\int_1^r
\dfrac{\sum_{z\in E_{t,B,4}^1}
\min
\,
\{\ord_z (h_i\circ f),3\}}{t}\,\der t
&
\,
\leq
\,
5
\int_1^r
\dfrac{\sum_{z\in E_{t,B,4}^1}
\min_{4\leq i\leq 6}
\,
\ord_z (h_i\circ f)}{t}\,\der t\\
&
\,
\leq
\,
\dfrac{5}{2}
\int_1^r
\dfrac{\sum_{z\in E_{t,B,4}^1}\ord_z (\mathsf{s}\circ f)}{t}\,\der t.
\end{align*}
Now, we proceed similarly as above
\begin{align*}
5
\,
T_f(r)
&
\,
\leq
\,
\sum_{i=1}^9
N_f^{[3]}(r,H_i)
+
S_f(r)
\\
&
\,
=
\,
\sum_{i=1}^3
\int_1^r
\dfrac{\sum_{z\in E_{t,A}}
\min
\,
\{\ord_z (h_i\circ f),3\}}{t}\,
\der t
+
\sum_{i=4}^6
\int_1^r
\dfrac{\sum_{z\in E_{t,B}}
\min
\,
\{\ord_z (h_i\circ f),3\}}{t}\,\der t
+
S_f(r)
\\
&
\,
=
\,
\sum_{i=1}^3
\bigg(
\int_1^r
\dfrac{\sum_{z\in E_{t,A,1}^1}
\min
\,
\{\ord_z (h_i\circ f),3\}}{t}\,\der t
+
\int_1^r
\dfrac{\sum_{z\in E_{t,A,1}^{\geq 2}}
\min
\,
\{\ord_z (h_i\circ f),3\}}{t}\,\der t
\bigg)\\
&
\,
+
\,
\sum_{i=4}^6
\bigg(
\int_1^r
\dfrac{\sum_{z\in E_{t,B,1}^1}
\min
\,
\{\ord_z (h_i\circ f),3\}}{t}\,\der t
+
\int_1^r
\dfrac{\sum_{z\in E_{t,B,4}^{\geq 2}}
\min
\,
\{\ord_z (h_i\circ f),3\}}{t}\,\der t
\bigg)
+
S_f(r)
\\
&
\,
\leq
\,
\dfrac{5}{2}
\,
\bigg(
\int_1^r
\dfrac{\sum_{z\in E_{t,A}}\ord_z (\mathsf{s}\circ f)}{t}\,\der t
+
\int_1^r
\dfrac{\sum_{z\in E_{t,B}}\ord_z (\mathsf{s}\circ f)}{t}\,\der t
\bigg)
\\
&
\,
+
\,
\bigg(
2-
\dfrac{5}{2}
\bigg)
\bigg(
\int_1^r
\dfrac{\sum_{z\in E_{t,A,1}^{\geq 2}}\ord_z (\mathsf{s}\circ f)}{t}\,\der t
+
\int_1^r
\dfrac{\sum_{z\in E_{t,B,4}^{\geq 2}}\ord_z (\mathsf{s}\circ f)}{t}\,\der t
\bigg)
+
S_f(r)
\\
&
\,
\leq
\,
\dfrac{5}{2}
\,
N_f(r,\mathcal{S})
-
\dfrac{1}{2}
\,
\bigg(
\int_1^r
\dfrac{\sum_{z\in E_{t,A,1}^{\geq 2}}\ord_z (\mathsf{s}\circ f)}{t}\,\der t
+
\int_1^r
\dfrac{\sum_{z\in E_{t,B,4}^{\geq 2}}\ord_z (\mathsf{s}\circ f)}{t}\,\der t
\bigg)
+
S_f(r)
\\
&
\,
\leq
\,
5
\,
T_f(r)
-
\dfrac{1}{2}
\,
\bigg(
\int_1^r
\dfrac{\sum_{z\in E_{t,A,1}^{\geq 2}}\ord_z (\mathsf{s}\circ f)}{t}\,\der t
+
\int_1^r
\dfrac{\sum_{z\in E_{t,B,4}^{\geq 2}}\ord_z (\mathsf{s}\circ f)}{t}\,\der t
\bigg)
+
S_f(r).
\end{align*}
This implies
\begin{align*}
\int_1^r
\dfrac{\sum_{z\in E_{t,A,1}^{\geq 2}}\ord_z (\mathsf{s}\circ f)}{t}\,\der t
&
\,
=
\,
S_f(r),\\
\int_1^r
\dfrac{\sum_{z\in E_{t,B,4}^{\geq 2}}\ord_z (\mathsf{s}\circ f)}{t}\,\der t
&
\,
=
\,
S_f(r),\\
\sum_{i=1}^6
N_f^{[3]}(r,H_i)
&
\,
=
\,
5
\,
T_f(r)
+
S_f(r),\\
\sum_{i=1}^3
N_f^{[3]}(r,H_i)
&
\,
=
\,
\dfrac{5}{2}
\int_1^r
\dfrac{\sum_{z\in E_{t,A}}\ord_z (\mathsf{r}\circ f)}{t}\,\der t
+
S_f(r),\\
\sum_{i=4}^6
N_f^{[3]}(r,H_i)
&
\,
=
\,
\dfrac{5}{2}
\int_1^r
\dfrac{\sum_{z\in E_{t,B}}\ord_z (\mathsf{r}\circ f)}{t}\,\der t
+
S_f(r).
\end{align*}
By proceeding similarly as in \eqref{-equation-h1h2h3,A}, we receive
\begin{align*}
\sum_{i=1}^3 N_f^{[3]}(r,H_i)
&
\,
=
\,
5
\,
N_f(r,A)
+
S_f(r),
\\
\sum_{i=4}^6 N_f^{[3]}(r,H_i)
&
\,
=
\,
5\,
N_f(r,B)
+
S_f(r).
\end{align*}
Hence, the equality \eqref{-equation-counting-A,B-order-case-b31} also holds in this case. Claim \ref{-claim} is thus proved.
\end{proof}
Next, since $f$ is of the form~\eqref{-form-of-f-case-2}, one can find a plane $\mathcal{K}=\{\mathsf{k}=0\}$ passing through $A$ and $C$ such that
\[
\ord_z(\mathsf{k}\circ f)
\,
\geq
\,
2\quad\quad{\scriptstyle(z\,\in \,f^{-1}(A))}.
\]
Let $\mathcal{B}_i=\{\mathsf{b}_i=0\}$ be the plane containing the two lines $AB$, $H_i\cap \mathcal{K}$ ($1\leq i\leq 3$). Since $\mathsf{b}_i$ is a linear combination of $h_i$ and $\mathsf{k}$ with nonzero coefficients, we have 
\[
\ord_z (\mathsf{b}_i\circ f)
\,
\geq
\,
2\quad\quad{\scriptstyle(z\,\in\, E_{t,A,i}^{\geq 2})},
\]
which yields
\begin{equation}
\label{-estimate-counting-h1,h2,h3-A-b1,b2,b3-E'tAi}
\sum_{i=1}^3
\ord_z(\mathsf{b}_i\circ f)
\geq
 4\quad\quad {\scriptstyle(z\,\in\,\cup_{i=1}^3\,E_{t,A,i}^{\geq 2})}.
\end{equation}
Let $\mathcal{C}=\{\mathsf{c}=\mathsf{b}_1\,\mathsf{b}_2\,\mathsf{b}_3=0\}$ be the degenerate cubic $\cup_{1\leq i\leq 3}\,\mathcal{B}_i$. It follows from~\eqref{-order-equal-1} and~\eqref{-estimate-counting-h1,h2,h3-A-b1,b2,b3-E'tAi} that
\begin{align*}
\min_{1\leq i\leq 3}
\,
\ord_z (h_i\circ f)
\,
&=
\,
1
\,
\leq
\,
\dfrac{1}{4}
\sum_{i=1}^3
\ord_z (\mathsf{b}_i\circ f)
\,
=
\,
\dfrac{1}{4}
\ord_z(\mathsf{c}\circ f)
\quad\quad {\scriptstyle(z\,\in\,\cup_{i=1}^3\,E_{t,A,i}^{\geq 2})},\\
\min_{1\leq i\leq 3}
\,
\ord_z(h_i\circ f)
\,
&=
\,
1
\,
\,
\leq
\,
\dfrac{1}{3}
\sum_{i=1}^3
\ord_z(\mathsf{b}_i\circ f)
\,
=
\,
\dfrac{1}{3}
\ord_z(\mathsf{c}\circ f)\quad\quad {\scriptstyle(z\,\in\,E_{t,A}\setminus\,\cup_{i=1}^3\,E_{t,A,i}^{\geq 2})},
\\
\min_{4\leq i\leq 6}
\,
\ord_z(h_i\circ f)
\,
&=
\,
1
\,
\,
\leq
\,
\dfrac{1}{3}
\sum_{i=1}^3
\ord_z(\mathsf{b}_i\circ f)
\,
=
\,
\dfrac{1}{3}
\ord_z(\mathsf{c}\circ f)\quad\quad {\scriptstyle(z\,\in\,E_{t,B})}.
\end{align*}
By taking the sum on disks and by integrating these inequalities,
\begin{align*}
\int_1^r
\dfrac{\sum_{z\in\cup_{i=1}^3
E_{t,A,i}^{\geq 2}}\min_{1\leq i\leq 3}
\,
\ord_z(h_i\circ f)}{t}\,\der t
&
\,
\leq
\,
\dfrac{1}{4}
\int_1^r
\dfrac{\sum_{z\in\cup_{i=1}^3 E_{t,A,i}^{\geq 2}}
\ord_z(\mathsf{c}\circ f)}{t}\,\der t,\\
\int_1^r
\dfrac{\sum_{z\in E_{t,A}\setminus\cup_{i=1}^3 E_{t,A,i}^{\geq 2}}
\min_{1\leq i\leq 3}
\,
\ord_z(h_i\circ f)}{t}\,\der t
&\,
\leq
\,
\dfrac{1}{3}
\int_1^r
\dfrac{\sum_{z\in E_{t,A}\setminus\cup_{i=1}^3 E_{t,A,i}^{\geq 2}}
\ord_z(\mathsf{c}\circ f)}{t}\,\der t,\\
N_f(r,B)
\,
=
\,
\int_1^r
\dfrac{\sum_{z\in E_{t,B}}
\min_{4\leq i\leq 6}
\,
\ord_z(h_i\circ f)}{t}\,\der t
&
\,
\leq
\,
\dfrac{1}{3}
\int_1^r
\dfrac{\sum_{z\in E_{t,B}}
\,
\ord_z(\mathsf{c}\circ f)}{t}\,\der t.
\end{align*}
By using these inequalities together with \eqref{-equation-counting-A,B-order-case-b31}, we receive
\begin{align*}
5
\,
T_f(r)
&
\,
=
\,
5
\,
N_f(r,A)
+
5
\,
N_f(r,B)
+
S_f(r)
\\
&
\,
=
\,
5
\,
\bigg(
\int_1^r
\dfrac{\sum_{z\in\cup_{i=1}^3 E_{t,A,i}^{\geq 2}}
\min_{1\leq i\leq 3}
\,
\ord_z(h_i\circ f)}{t}\,\der t
+
\int_1^r
\dfrac{\sum_{z\in E_{t,A}\setminus\cup_{i=1}^3 E_{t,A,i}^{\geq 2}}
\min_{1\leq i\leq 3}
\,
\ord_z(h_i\circ f)}{t}\,\der t
\bigg)
\\
&
\,
+
\,
5
\int_1^r
\dfrac{\sum_{z\in E_{t,B}}
\min_{4\leq i\leq 6}
\,
\ord_z(h_i\circ f)}{t}\,\der t
+
S_f(r)
\\
&
\,
\leq
\,
\dfrac{5}{3}
\,
\bigg(
\int_1^r
\dfrac{\sum_{z\in E_{t,A}}\ord_z(\mathsf{c}\circ f)}{t}\,\der t
+
\int_1^r
\dfrac{\sum_{z\in E_{t,B}}\ord_z (\mathsf{c}\circ f)}{t}\,\der t
\bigg)
\\
&
\,
+
\,
\bigg(
\dfrac{5}{4}
-
\dfrac{5}{3}
\bigg)
\int_1^r
\dfrac{\sum_{z\in\cup_{i=1}^3 E_{t,A,i}^{\geq 2}}
\ord_z(\mathsf{c}\circ f)}{t}\,\der t
+
S_f(r)
\\
&
\,
\leq
\,
\dfrac{5}{3}
\,
N_f(r,\mathcal{C})
-
\dfrac{5}{12}
\int_1^r
\dfrac{\sum_{z\in\cup_{i=1}^3 E_{t,A,i}^{\geq 2}}\ord_z(\mathsf{c}\circ f)}{t}\,\der t
+
S_f(r)
\\
&
\,
\leq
\,
5
\,
T_f(r)
-
\dfrac{5}{12}
\int_1^r
\dfrac{\sum_{z\in\cup_{i=1}^3 E_{t,A,i}^{\geq 2}}\ord_z(\mathsf{c}\circ f)}{t}\,\der t
+
S_f(r).
\end{align*}
This implies
\begin{equation}
\label{-tangent-not-often-A}
\int_1^r
\dfrac{\sum_{z\in\cup_{i=1}^3 E_{t,A,i}^{\geq 2}}
\ord_z(\mathsf{c}\circ f)}{t}\,\der t
\,
=
\,
S_f(r).
\end{equation}
By using \eqref{-first-estimate-ord_z-h_i} and \eqref{-estimate-counting-h1,h2,h3-A-b1,b2,b3-E'tAi}, we get
\[
\sum_{i=1}^3
\min
\,
\{\ord_z (h_{i}\circ f),3\}
\,
\leq
\,
6
\,
\leq
\,
\dfrac{3}{2}
\ord_z(\mathsf{c}\circ f)\quad\quad{\scriptstyle (z\,\in\cup_{i=1}^3 E_{t,A,i}^{\geq 2})},
\]
which yields 
\begin{align}
\label{-small-counting-1}
\int_1^r
\dfrac{\sum_{z\in\cup_{i=1}^3 E_{t,A,i}^{\geq 2}}
\min_{1\leq i\leq 3}
\,
\ord_z(h_i\circ f)}{t}\,\der t
&
\,
\leq
\,
\dfrac{3}{2}
\int_1^r
\dfrac{\sum_{z\in\cup_{i=1}^3 E_{t,A,i}^{\geq 2}}\ord_z(\mathsf{c}\circ f)}{t}\,\der t\notag\\
\explain{Use \eqref{-tangent-not-often-A}}
\ \ \ \ \ \ \ \ \ \ \ \ \ \ \ \ \ \ \ \ \ \ \ \ \ \ \ \ \ \ \ 
&
\,
=
\,
S_f(r).
\end{align}
Moreover, we also have
\[
\sum_{i=1}^3
\min
\,
\{\ord_z (h_{i}\circ f),3\}
\,
=
\,
3
\,
=
\,
3
\,
\min_{1\leq i\leq 3}
\,
\ord_z(h_i\circ f)\quad\quad{\scriptstyle (z\,\in\, E_{t,A}\,\setminus\,\cup_{i=1}^3 E_{t,A,i}^{\geq 2})},
\]
which implies, by integrating, that
\begin{equation}
\label{-counting-h1,h2,h3,counting-at-A,set-outside-E'tA,i}
\sum_{i=1}^3
\int_1^r
\dfrac{\sum_{z\in E_{t,A}\setminus\cup_{i=1}^3 E_{t,A,i}^{\geq 2}}
\min
\,
\{\ord_z (h_i\circ f),3\}}{t}\,\der t
\,
\leq
\,
3
\int_1^r
\dfrac{\sum_{z\in E_{t,A}\setminus\cup_{i=1}^3 E_{t,A,i}^{\geq 2}}\min_{1\leq i\leq 3}
\,
\ord_z(h_i\circ f)}{t}\,\der t.
\end{equation}
By combining \eqref{-small-counting-1} and \eqref{-counting-h1,h2,h3,counting-at-A,set-outside-E'tA,i}, we get
\begin{align*}
\sum_{i=1}^3
N_f^{[3]}(r,H_i)
&
\,
=
\,
\sum_{i=1}^3
\int_1^r
\dfrac{\sum_{z\in E_{t,A}}
\min
\,
\{\ord_z (h_i\circ f),3\}}{t}\,\der t\\
&
\,
=
\,
\sum_{i=1}^3
\bigg(
\int_1^r
\dfrac{\sum_{z\in E_{t,A}\setminus\cup_{i=1}^3 E_{t,A,i}^{\geq 2}}
\min
\,
\{\ord_z (h_i\circ f),3\}}{t}\,\der t\\
&
\,
+
\,
\int_1^r
\dfrac{\sum_{z\in\cup_{i=1}^3 E_{t,A,i}^{\geq 2}}
\min
\,
\{\ord_z (h_i\circ f),3\}}{t}\,\der t
\bigg)
+
S_f(r)
\\
&
\,
\leq
\,
3
\int_1^r
\dfrac{\sum_{z\in E_{t,A}\setminus\cup_{i=1}^3 E_{t,A,i}^{\geq 2}}
\min_{1\leq i\leq 3}
\,
\ord_z(h_i\circ f)}{t}\,\der t
+
S_f(r)\\
&
\,
\leq
\,
3
\,
N_f(r,A)
+
S_f(r).
\end{align*}
By symmetry, we also have
\begin{equation*}
\label{-estimate-counting-h4,h5,h6-tangent-often}
\sum_{i=4}^6
N_f^{[3]}(r,H_i)
\,
\leq
\,
3
\,
N_f(r,B)
+
S_f(r).
\end{equation*}
Hence we can rewrite \eqref{-estimate-order-counting-case-b31} to get a contradiction:
\begin{align*}
5
\,
T_f(r)
&
\,
\leq
\,
\sum_{i=1}^{9}
N_f^{[3]}(r,H_i)
+
S_f(r)\\
&
\,
\leq
\,
3
\,
N_f(r,A)
+
3\,N_f(r,B)
+
S_f(r)\\
&
\,
\leq
\,
3
\,N_f(r,\mathcal{B})
+
S_f(r)\\
&
\,
\leq
\,
3
\,
T_f(r)
+
S_f(r).
\end{align*}
\end{proof}

In $\mathbb{P}^4(\mathbb{C})$, by the generic condition for the family of hyperplanes $\{H_i\}_{1\leq i\leq q}$, when $q\geq 10$, we see that, for all three disjoint subsets $I$, $J$, $K$ of $\{1,\ldots,q\}$ with $|I|\geq 2$, $|J|\geq 2$, $|I|+|J|=6$, $|K|=4$, the diagonal hyperplane $H_{IJ}$ does not contain the point $\cap_{k\in K}H_k$.

\begin{lem}\label{start_lem_P4}
In $\mathbb{P}^4(\mathbb{C})$, all complements of the form \eqref{hyperbolic_complement} are hyperbolic if $m=1$.
\end{lem}

\begin{proof}
We can assume that $A_{1,4}$ is a set consisting of one element in
\[
\big(
\cup_{1\leq i_1<i_2\leq 10}\,
(
H_{i_1}
\cap 
H_{i_2}
)^*
\big)\,
\bigcup\,
\big(
\cup_{1\leq i_1<i_2<i_3\leq 10}\,
H_{i_1}
\cap 
H_{i_2}
\cap H_{i_3}
\big)^*
\bigcup\,
\big(
\cup_{1\leq i_1<i_2<i_3<i_4\leq 10}\,
H_{i_1}
\cap 
H_{i_2}
\cap H_{i_3}
\cap H_{i_4}
\big).
\]
Suppose to the contrary that there is an entire curve $f\colon\mathbb{C}\rightarrow\mathbb{P}^4(\mathbb{C})\setminus\big(\cup_{i=1}^{10}H_i\setminus A_{1,4}\big)$. If $A_{1,4}$ is not a set of a point, then $f$ avoids at least seven hyperplanes. By Theorem~\ref{diagonal_streng}, its image is contained in a line $L$ and we can continue to analyze the position of $L$ with respect to $\cup_{i=1}^{10}H_i\setminus A_{1,4}$ to get a contradiction. Consider the remaining case where $A_{1,4}$ consists of a point, say $\cap_{i=1}^4\,H_i$. By Theorem \ref{diagonal}, the curve $f$ lands in some diagonal hyperplane of the family $\{H_i\}_{5\leq i\leq 10}$, which does not contain the point $\cap_{i=1}^4 H_i$ by the generic condition. Hence, $f$ must avoid all $H_i$ ($1\leq i\leq 10$), which is impossible by Theorem~\ref{fujimoto-green}.
\end{proof}
\subsection{Stability of intersections}

We will also invoke the following known complex analysis fact.

\begin{namedthm*}{Stability of intersections}\label{stability}
 Let $X$ be a complex manifold and let $H \subset X$ be an analytic hypersurface. Suppose that a sequence $(f_n)$ of entire curves in $X$ converges toward an entire curve $f$. If $f(\mathbb{C})$ is not contained in $H$, then we have
\[
f(\mathbb{C})
\cap H
\subset
\lim f_n(\mathbb{C})\cap H.
\]
\end{namedthm*}

\section{Proof of the Main Theorem}

We keep the notation of the previous section. Let $S$ be a hypersurface of degree $2n$, which is in general position with respect to the family $\{H_i\}_{1\leq i\leq 2n}$. We would like to determine what conditions $S$ should satisfy for $\Sigma_{\epsilon}$ to be hyperbolic. Suppose that $\Sigma_{\epsilon_k}$ is not hyperbolic for a sequence $(\epsilon_k)$ converging to $0$. Then we can find entire curves $f_{\epsilon_k}\colon\mathbb{C}\rightarrow\Sigma_{\epsilon_k}$. By the Brody lemma, after reparameterization and extraction, we may assume that the sequence $(f_{\epsilon_k})$ converges to an entire curve $f\colon\mathbb{C}\rightarrow \cup_{i=1}^{2n}H_i$. The curve $f(\mathbb{C})$ lands inside some hyperplane $H_i$. Moreover, it cannot land inside any subspace of dimension 1 (a line). Indeed, if $f(\mathbb{C})\subset \cap_{i\in I} H_i$ for some subset $I$ of the index set $\mathbf{Q}=\{1,\ldots,2n\}$ having cardinality $n-1$, then for all $j\in \mathbf{Q}\setminus I$, by stability of intersections, one has
\[
f(\mathbb{C})\cap H_j
\subset
\lim f_{\epsilon_k}(\mathbb{C})\cap H_j\subset
\lim\Sigma_{\epsilon_k}\cap H_j\subset S\cap H_j.
\]
Thus $f(\mathbb{C})$ and $H_j$ have empty intersection by the general position. Hence the curve $f(\mathbb{C})$ lands in
\[
\cap_{i\in I}H_i\setminus\big(\cup_{j\in \mathbf{Q}\setminus I} H_j\big).
\]
This is a contradiction, because the complement of $n+1$ ($n\geq 3$) points in a line is hyperbolic by Picard's theorem.

Now, let $I$ be the largest subset of $\mathbf{Q}$ such that the curve $f(\mathbb{C})$ lands in $\cap_{i\in I}H_i$. We have $|I|\leq n-2$. By stability of intersections, $f(\mathbb{C})\cap H_j$ is contained in $S$ for all $l\in \mathbf{Q}\setminus I$. Therefore the curve $f(\mathbb{C})$ lands in
\begin{equation}\label{reduced_problem}
\cap_{i\in I}H_i\setminus\big(\cup_{j\in \mathbf{Q}\setminus I}H_j\setminus S\big).
\end{equation}
So, the problem reduces to finding a hypersurface $S$ of degree $2n$ such that all complements of the form \eqref{reduced_problem} are hyperbolic, where $I$ is a subset of $\mathbf{Q}$ of cardinality at most $n-2$. For example when $n=3$ (\cite{duval2014}), we need to find a sextic curve $S$ such that all complements of the form $H_i\setminus\big(\cup_{j\not=i}H_j\setminus S\big)$ are hyperbolic. In this case, we have the complement of five lines in the hyperplane $H_i$ on which all points of intersection with $S$ are deleted.

\begin{center}
\begin{picture}(0,0)%
\includegraphics{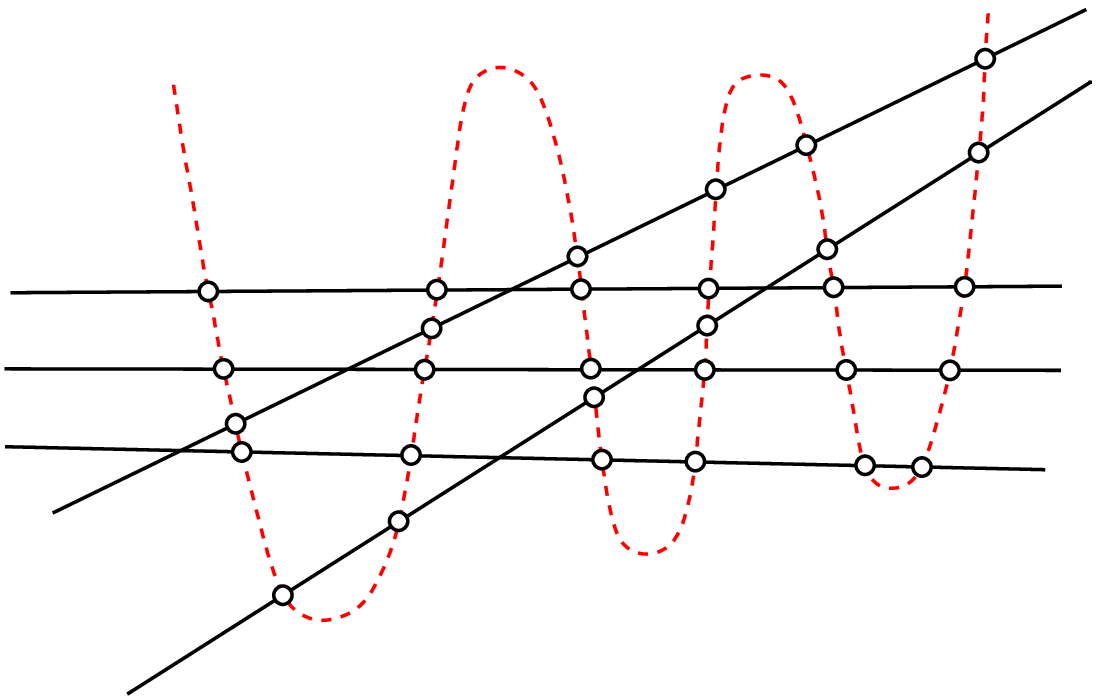}%
\end{picture}%
\setlength{\unitlength}{4144sp}%
\begingroup\makeatletter\ifx\SetFigFont\undefined%
\gdef\SetFigFont#1#2#3#4#5{%
  \reset@font\fontsize{#1}{#2pt}%
  \fontfamily{#3}\fontseries{#4}\fontshape{#5}%
  \selectfont}%
\fi\endgroup%
\begin{picture}(5007,3174)(637,-3077)
\end{picture}%
\end{center}

We will construct such $S$ by deformation, step by step. For $2\leq l\leq n-1$, let $\Delta_l$ be a finite collection of subspaces of dimension $n-l$, in the sense of section~\ref{-starting-lemmas}. Let $D_l\not \in \Delta_l$ be another subspace of dimension $n-l$, defined as $D_l=\cap_{i\in I_{D_l}}H_i$. For a hypersurface $S=\{s=0\}$ in general position with respect to the family $\{H_i\}_{1\leq i\leq 2n}$ and $\epsilon\not=0$, we set
\[
S_{\epsilon}
=
\big\{\epsilon s+\Pi_{i\not\in I_{D_l}}h_i^{n_i}=0\big\},
\]
where $n_i\geq 1$ are chosen (freely) so that $\sum_{i\not\in I_{D_l}}n_i=2n$. It is not hard to see that the hypersurface $S_{\epsilon}$ is also in general position with respect to the family $\{H_i\}_{1\leq i\leq 2n}$. We denote by $\overline{\Delta}_l$ the family of all subspaces of dimension $n-l$ ($2\leq l\leq n$) with the convention $\overline{\Delta}_n=\emptyset$.

\begin{lem}\label{lem_deformation} Assume that all complements of the form
\begin{equation}\label{start_form}
\cap_{i\in I}H_i\setminus\big(\cup_{j\in J}H_j\setminus(((\Delta_l\cup\overline{\Delta}_{l+1})\cap S)\cup A_{m,n-|I|})\big)
\end{equation}
are hyperbolic where $I,J$ are two disjoint subsets of $\{1,\dots, 2n\}$ such that $|I|\leq n-2$, $|J|+2|I|\geq 2n+1$ and $m\leq|J|+2|I|- (2n+1)$. Here, $A_{m,n-|I|}$ is a set of at most $m$ star-subspaces coming from the family of hyperplanes $\{\cap_{i\in I}H_i\cap H_j\}_{j\in J}$ in $\cap_{i\in I}H_i\cong\mathbb{P}^{n-|I|}(\mathbb{C})$. Then all complements of the form
\begin{equation}\label{end_form}
\cap_{i\in I}H_i\setminus\big(\cup_{j\in J}H_j\setminus(((\Delta_l\cup D_l\cup\overline{\Delta}_{l+1})\cap S_{\epsilon})\cup A_{m,n-|I|})\big)
\end{equation}
are also hyperbolic for sufficiently small $\epsilon\not=0$.
\end{lem}

\begin{proof}
By the definition of $S_{\epsilon}$, we see that $S_{\epsilon}\cap\big(\cap_{m\in M}H_m\big)= S\cap\big(\cap_{m\in M}H_m\big)$ when $M\cap(\mathbf{Q}\setminus I_{D_l})\not=\emptyset$, hence
\[
(\Delta_l\cup D_l\cup\overline{\Delta}_{l+1})\cap S_{\epsilon}=
((\Delta_l\cup \overline{\Delta}_{l+1})\cap S)\cup(D_l\cap S_{\epsilon}).
\]
When $|I|\geq l$, using this, we observe that two complements \eqref{start_form}, \eqref{end_form} coincide.

Assume therefore $|I|\leq l-1$. Suppose by contradiction that there exists a sequence of entire curves $(f_{\epsilon_k}(\mathbb{C}))_k$, $\epsilon_k\rightarrow 0$ contained in the complement
\[
\cap_{i\in I}H_i\setminus\big(\cup_{j\in J}H_j\setminus(((\Delta_l\cup D_l\cup\overline{\Delta}_{l+1})\cap S_{\epsilon_k})\cup A_{m,n-|I|})\big).
\]
By the Brody Lemma,  we may assume that $(f_{\epsilon_k})$ converges to an entire curve $f(\mathbb{C})\subset \cap_{i\in I}H_i$. Our aim is to prove that the curve $f(\mathbb{C})$ lands in some complement of the form \eqref{start_form}. Let $\cap_{k\in K}H_k$ be the smallest subspace containing $f(\mathbb{C})$. It is clear that $K\supset I$. Take an index $j$ in $J\setminus K$. By stability of intersections, one has
\begin{align}\label{intersection}
f(\mathbb{C})\cap H_j& \subset \lim f_{\epsilon_k}(\mathbb{C})\cap H_j\notag\\
&\subset ((\Delta_l\cup \overline{\Delta}_{l+1})\cap S)\cup A_{m,n-|I|}\cup\lim(D_l\cap S_{\epsilon_k}).
\end{align}
If the index $j$ does not belong to $I_{D_l}$, then $H_j\cap D_l\cap S_{\epsilon_k}\subset \overline{\Delta}_{l+1}\cap S$. It follows from \eqref{intersection} that
\begin{equation}\label{intersection1}
f(\mathbb{C})\cap H_j
\subset ((\Delta_l\cup \overline{\Delta}_{l+1})\cap S)\cup A_{m,n-|I|}.
\end{equation}
If the index $j$ belongs to $I_{D_l}$, noting that $\lim (D_l\cap S_{\epsilon_k})$ is contained in $D_l\cap (\cup_{i\not\in I_{D_l}} H_i)$, again from \eqref{intersection}, one has
\begin{equation}\label{intersection2}
f(\mathbb{C})\cap H_j
\subset
((\Delta_l\cup \overline{\Delta}_{l+1})\cap S)\cup A_{m,n-|I|}\cup (D_l\cap(\cup_{i\not\in I_{D_l}}H_i)).
\end{equation}

Assume first that $K=I$. We claim that \eqref{intersection1} also holds when the index $j\in J\setminus I$ belonging to $I_{D_l}$. Indeed, for the supplementary part in \eqref{intersection2}, we have
\[
f(\mathbb{C})\cap H_j\cap
\big(
D_l\cup_{i\not\in I_{D_l}}
H_i
\big)
\subset
\cup_{i\not\in I_{D_l}}
(f(\mathbb{C})\cap H_j\cap H_i),
\]
so that \eqref{intersection1} applies here to all $i\not\in I_{D_l}$. Hence, the curve $f(\mathbb{C})$ lands inside
\[
\cap_{i\in I}H_i
\setminus
\big(
\cup_{j\in J}H_j\setminus(((\Delta_l\cup\overline{\Delta}_{l+1})\cap S)\cup A_{m,n-|I|})\big),
\]
contradicting the hypothesis.\\

Assume now that $I$ is a proper subset of $K$. Let us set
\[
A_{m,n-|I|,K}
=
\{X\cap(\cap_{k\in K}H_k)|X\in A_{m,n-|I|}\}.
\]
This set consists of star-subspaces of $\cap_{k\in K}H_k\cong\mathbb{P}^{n-|K|}(\mathbb{C})$. Let $B_{m,K}$ be the subset of $A_{m,n-|I|,K}$ containing all star-subspaces of dimension $n-|K|-1$ (i.e., of codimension $1$ in $\cap_{k\in K}H_k$), and let $C_{m,K}$ be the remaining part. A star-subspace in $B_{m,K}$ is of the form $(\cap_{k\in K}H_k\cap H_j)^*$ for some index $j\in J\setminus K$. Then let $R$ denote the set of such indices $j$, so that
\[
|R|=|B_{m,K}|.
\]
We consider two cases separately, depending on the dimension of the subspace $Y=\cap_{k\in K}H_k\cap D_l$.
\begin{itemize}
\item[\textbf{Case 1.}] $Y$ is a subspace of dimension $n-|K|-1$. In this case, $Y$ is of the form $(\cap_{k\in K}H_k)\cap H_y$ for some index $y$ in $I_{D_l}$. It follows from \eqref{intersection}, \eqref{intersection1}, \eqref{intersection2} that the curve $f(\mathbb{C})$ lands inside the set
\[
\cap_{k\in K}H_k
\setminus\big(\cup_{j\in (J\setminus K)\setminus (R\cup\{y\})}H_j\setminus(((\Delta_l\cup\overline{\Delta}_{l+1})\cap S)\cup C_{m,K})\big).
\]
Now we need to show that this set is of the form \eqref{start_form}. First, we verify the corresponding required inequality between cardinalities
\begin{align*}
|(J\setminus K)\setminus (R\cup\{y\})|
&
\,
\geq
\,
|J\setminus K|-|B_{m,K}|-1\\
&
\,
\geq
\,
|J|-|J\cap K|-\big(|J|+2|I|-2n-1-|C_{m,K}|\big)-1\\
&
\,
=
\,
2\,
(n-|K|)+|C_{m,K}|+ 2|K\setminus I|-|J\cap K|\\
&
\,
\geq
\,
2
\,
(n-|K|)+1+|C_{m,K}|,
\end{align*}
where the last inequality holds because $I$ and $J$ are two disjoint sets and $I$ is a proper subset of $K$. Secondly, we verify that the set $K$ is of cardinality at most $n-2$. Indeed, if $|K|=n-1$, then since $S$ is in general position with respect to $\{H_i\}_{1\leq i\leq 2n}$, we see that
\[
\cap_{k\in K}H_k
\setminus\big(\cup_{j\in (J\setminus K)\setminus (R\cup\{y\})}H_j\setminus(((\Delta_l\cup\overline{\Delta}_{l+1})\cap S)\cup C_{m,K})\big)
=
\cap_{k\in K}H_k
\setminus\big(\cup_{j\in (J\setminus K)\setminus (R\cup\{y\})}H_j\setminus C_{m,K}\big).
\]
Owing to the inequality $|(J\setminus K)\setminus (R\cup\{y\})|\geq 3+|C_{m,K}|$, the curve $f$ lands in a complement of at least three points in a line. By Picard's theorem, $f$ is constant, which is a contradiction.
\item[\textbf{Case 2.}] $Y$ is a subspace of dimension at most $n-|K|-2$. In this case, the curve $f(\mathbb{C})$ lands inside
\[
\cap_{k\in K}H_k
\setminus
\big(
\cup_{j\in (J\setminus K)\setminus R}H_j
\setminus
(((\Delta_l\cup\overline{\Delta}_{l+1})\cap S)\cup C_{m,K}\cup Y^*)
\big).
\]
This set is of the form \eqref{start_form} since
\[
|\{j\in(J\setminus K)\setminus R\}|
\,
\geq
\,
2
\,
(n-|K|)+1+|C_{m,K}\cup Y^*|,
\]
which also implies $|K|\leq n-2$ by similar argument as in \textbf{Case 1}.
\end{itemize}
The lemma is thus proved.
\end{proof}

\begin{proof}[End of proof of the Main Theorem]
We now come back to the proof of the Main Theorem. Keep the notation as in Lemma~\ref{lem_deformation}. We claim that $\{\cap_{i\in I}H_i\cap H_j\}_{j\in J}$ is also a family of generic hyperplanes in the projective space $\cap_{i\in I}H_i\cong\mathbb{P}^{n-|I|}(\mathbb{C})$. Indeed, let $\mathcal{I}$, $\mathcal{J}$, $\mathcal{J}_1,\ldots,\mathcal{J}_k$ be disjoint subsets of $J$ such that $|\mathcal{I}|, |\mathcal{J}_i|\geq 2$, $|\mathcal{I}|+|\mathcal{J}_i|=(n-|I|)+2$, $1\leq i\leq k$ and let $\{i_1,\ldots,i_l\}$ be a subset of $\mathcal{I}$. Let us set $\mathbf{I}=I\cup\mathcal{I}$; then the intersection between the $|\mathcal{J}|$ hyperplanes $H_j,j\in \mathcal{J}$, the $k$ diagonal hyperplanes $H_{\mathbf{I}\mathcal{J}_1},\ldots,H_{\mathbf{I}\mathcal{J}_k}$, and the $|I|+l$ hyperplanes $H_i$ ($i\in I$), $H_{i_1},\ldots,H_{i_l}$ is a linear subspace of codimension $\min\{k+|I|+l,|\mathbf{I}|\}+|\mathcal{J}|$, with the convention that when $\min\{k+|I|+l,|\mathbf{I}|\}+|\mathcal{J}|>n$, this intersection is empty. Since
\[
\min\{k+|I|+l,|\mathbf{I}|\}+|\mathcal{J}|
=
\min\{k+l,|\mathcal{I}|\}+|I|+|\mathcal{J}|
\]
we deduce that in the projective space $\cap_{i\in I}H_i$, the intersection between the $|\mathcal{J}|$ hyperplanes $H_j,j\in \mathcal{J}$, the $k$ diagonal hyperplanes $H_{\mathcal{I}\mathcal{J}_1},\ldots,H_{\mathcal{I}\mathcal{J}_k}$, and the $l$ hyperplanes $H_{i_1},\ldots,H_{i_l}$ is a linear subspace of codimension $\min\{k+l,|\mathcal{I}|\}+|\mathcal{J}|$, with the convention that when $\min\{k+l,|\mathcal{I}|\}+|\mathcal{J}|>n-|I|$, this intersection is empty.

\textbf{Starting point of the process by deformation:} We start with the hyperbolicity of all complements of the forms
\[
\cap_{i\in I}H_i\setminus\big(\cup_{j\in J}H_j\setminus A_{m,n-|I|}\big),
\]
where $I$, $J$, $A_{m,n-|I|}$ are as in Lemma~\ref{lem_deformation}. More precisely,
\begin{itemize}
\item[$\bullet$] when $n=3$, we start with the hyperbolicity of all complements $H_i\setminus\big(\cup_{j\not=i}H_j\big)$, which follows from Theorem~\ref{fujimoto-green} in $\mathbb{P}^2(\mathbb{C})$;
\item[$\bullet$] when $n=4$, we start with the hyperbolicity of all complements
\begin{align*}
& H_i\setminus\big(\cup_{j\not=i}H_j\big),\\
&\cap_{i\in I}H_i\setminus\big(\cup_{j\in J}H_j\setminus A_{1,2}\big)\quad\quad {\scriptstyle(|I|\,=\,2\,,\,5+|A_{1,2}|\,\leq\,|J|\,\leq\,6)},
\end{align*}
which follows from Theorem~\ref{fujimoto-green} in $\mathbb{P}^3(\mathbb{C})$ and Lemma \ref{start_lem_P2} for $m=1$;
\item[$\bullet$] when $n=5$, we start with the hyperbolicity of all complements
\begin{align*}
& H_i\setminus\big(\cup_{j\not=i}H_j\big),\\
&\cap_{i\in I}H_i\setminus\big(\cup_{j\in J}H_j\setminus A_{1,3}\big)\quad\quad {\scriptstyle(|I|\,=\,2\,,\,7+|A_{1,3}|\,\leq\,|J|\,\leq\,8)},\\
&\cap_{i\in I}H_i\setminus\big(\cup_{j\in J}H_j\setminus A_{2,2}\big)\quad\quad {\scriptstyle(|I|\,=\,3\,,\,5+|A_{2,2}|\,\leq\,|J|\,\leq 7)},
\end{align*}
which follows from Theorem~\ref{fujimoto-green} in $\mathbb{P}^4(\mathbb{C})$, Lemma~\ref{start_lem_P3} for $m=1$, and Lemma~\ref{start_lem_P2} for $m=2$;
\item[$\bullet$] when $n=6$, we start with the hyperbolicity of all complements
\begin{align*}
& H_i\setminus\big(\cup_{j\not=i}H_j\big),\\
&\cap_{i\in I}H_i\setminus\big(\cup_{j\in J}H_j\setminus A_{1,4}\big)\quad\quad {\scriptstyle(|I|\,=\,2\,,\,9+|A_{1,4}|\,\leq\,|J|\,\leq\,10)},\\
&\cap_{i\in I}H_i\setminus\big(\cup_{j\in J}H_j\setminus A_{2,3}\big)\quad\quad {\scriptstyle(|I|\,=\,3\,,\,7+|A_{2,3}|\,\leq\,|J|\,\leq\,9)},\\
&\cap_{i\in I}H_i\setminus\big(\cup_{j\in J}H_j\setminus A_{3,2}\big)\quad\quad {\scriptstyle(|I|\,=\,4\,,\,5+|A_{3,2}|\,\leq\,|J|\,\leq\,8)},
\end{align*}
which follows from Theorem~\ref{fujimoto-green} in $\mathbb{P}^5(\mathbb{C})$, Lemma~\ref{start_lem_P4} for $m=1$, Lemma~\ref{start_lem_P3} for $m=2$, and Lemma~\ref{start_lem_P2} for $m=3$.
\end{itemize}

\textbf{Details of the process by deformation:} In the first step, we apply inductively Lemma \ref{lem_deformation} for $l=n-1$ and get at the end a hypersurface $S_1$ such that all complements of the forms
\begin{align*}
&\cap_{i\in I}H_i
\setminus
\big(
\cup_{j\in J}H_j\setminus(S_1\cup A_{m,n-|I|})
\big)\quad\quad\quad\quad\quad\quad\, {\scriptstyle(|I|\,=\,n-2)},\\
&\cap_{i\in I}H_i
\setminus
\big(
\cup_{j\in J}H_j\setminus((\overline{\Delta}_{n-1}\cap S_1)\cup A_{m,n-|I|})
\big)\quad\quad {\scriptstyle(|I|\,\leq\, n-3)}
\end{align*}
are hyperbolic. Considering this as the starting point of the second step, we apply inductively Lemma \ref{lem_deformation} for $l=n-2$. Continuing this process, we get at the end of the $(n-2)\text{th}$ step a hypersurface $S=S_{n-2}$ satisfying the required properties.
\end{proof}
\section{Some discussion}

Actually, our method works for a family of at least $2n$ generic hyperplanes in $\mathbb{P}^n(\mathbb{C})$. We hope that the Main Theorem is true for all $n\geq 3$. As we saw above, the problem reduces to proving the following conjecture.
\begin{namedthm*}{Conjecture}
All complements of the form \eqref{hyperbolic_complement} are hyperbolic.
\end{namedthm*}
We already know it to be true for $n=2$, since Lemma \ref{start_lem_P2} holds generally, without restriction on $m$.
\begin{lem}\label{conjecture_P2}
In $\mathbb{P}^2(\mathbb{C})$, all complements of the form \eqref{hyperbolic_complement} are hyperbolic
\end{lem}
\begin{proof}
Assume now $m\geq 4$ and $A_{m,2}=\{A_1,\ldots,A_m\}$, where $A_i=H_{i_1}\cap H_{i_2}$ ($1\leq i\leq m$). We denote by $I$ the index set $\{i_j\colon 1\leq i\leq m,1\leq j\leq 2\}$. Suppose to the contrary that there exists an entire curve $f\colon\mathbb{C}\rightarrow \mathbb{P}^2(\mathbb{C})\setminus (\cup_{i=1}^{5+m}H_i\setminus A_{m,2})$. By the generic condition, we can assume that $f$ is linearly nondegenerate. By similar arguments as in Lemma~\ref{start_lem_P2} (cf.~\eqref{-compare-the-counting-functions1}), we have
\begin{equation*}
\label{-estimate-counting-Hi,Ai}
\sum_{i\in I}
N_f^{[2]}(r,H_i)
\,
\leq
\,
3
\,
\sum_{i=1}^m
N_f(r,A_i).
\end{equation*}
Let $\mathcal{C}_m=\{\mathsf{c}_m=0\}$ be an algebraic curve in $\mathbb{P}^2(\mathbb{C})$ of degree $d$ passing through all points in $A_{m,2}$ with multiplicity at least $k$ which does not contain the curve $f(\mathbb{C})$.
Starting from the inequality
\[
\min_{1\leq j\leq 2}
\,
\ord_z(h_{i_j}\circ f)
\,
\leq
\,
\dfrac{1}{k}
\ord_z(\mathsf{c}_m\circ f)\quad\quad
{\scriptstyle(z\,\in\,f^{-1}(A_i))}
\]
and proceeding as in \eqref{-compare-the-counting-function-at-points-and-line-L}, we get
\begin{equation*}
\label{-estimate-couting-at-pointAi-couting-of-cm}
\sum_{i=1}^m
N_f(r,A_i)
\,
\leq
\,
\dfrac{1}{k}
\,
N_f(r,\mathcal{C}_m).
\end{equation*}
We may then proceed similarly as in \eqref{-last-estimate-conclude1}
\begin{align}
\label{-last-estimate1}
(m+2)
\,
T_f(r)
&
\,
\leq
\,
\sum_{i=1}^{5+m}
N_f^{[2]}(r,H_i)
+
S_f(r)
\notag\\
&
\,
\leq
\,
3
\sum_{i=1}^m
N_f(r,A_i)
+
S_f(r)\notag\\
&
\,
\leq
\,
\dfrac{3}{k}
\,
N_f(r,\mathcal{C}_m)
+
S_f(r)\notag\\
&
\,
\leq
\,
\dfrac{3d}{k}
\,
T_f(r)
+
S_f(r).
\end{align}
When $m\geq 5$, the following claim yields a concluding contradiction.
\begin{claim}If $m\geq 5$, we can find some curve $\mathcal{C}_m$ which does not contain $f(\mathbb{C})$ such that 
\begin{equation}\label{k,d relation 1}
k
\,
>
\,
\dfrac{3d}{m+2}.
\end{equation}
\end{claim}
Indeed, the degree of freedom for the choice of a curve of degree $d$ is
\[
\dfrac{(d+1)(d+2)}{2}-1.
\]
We want $\mathcal{C}_m$ to pass through all points in $A_{m,2}$ with multiplicity at least $k$. The number of equations (with the coefficients of $\mathcal{C}_m$ as unknowns) for this is not greater than
\begin{equation*}
\label{-condition}
m\,\dfrac{k(k+1)}{2}.
\end{equation*}
Thus, for the existence of $\mathcal{C}_m$, it is necessary that
\begin{equation}\label{k,d relation 2}
\dfrac{(d+1)(d+2)}{2}-1
\,
>\,
m
\,
\dfrac{k(k+1)}{2}.
\end{equation}
We try to find two natural numbers $k,d$ satisfying \eqref{k,d relation 1} and \eqref{k,d relation 2}. This can be done by choosing $d=(m+2)M$ and $k=3M+1$ for large enough $M$. Using the remaining freedom in the choice of $\mathcal{C}_m$, we can choose it not containing $f(\mathbb{C})$, which proves the claim.

Next, we consider the remaining case where $m=4$.
\begin{center}
\begin{picture}(0,0)%
\includegraphics{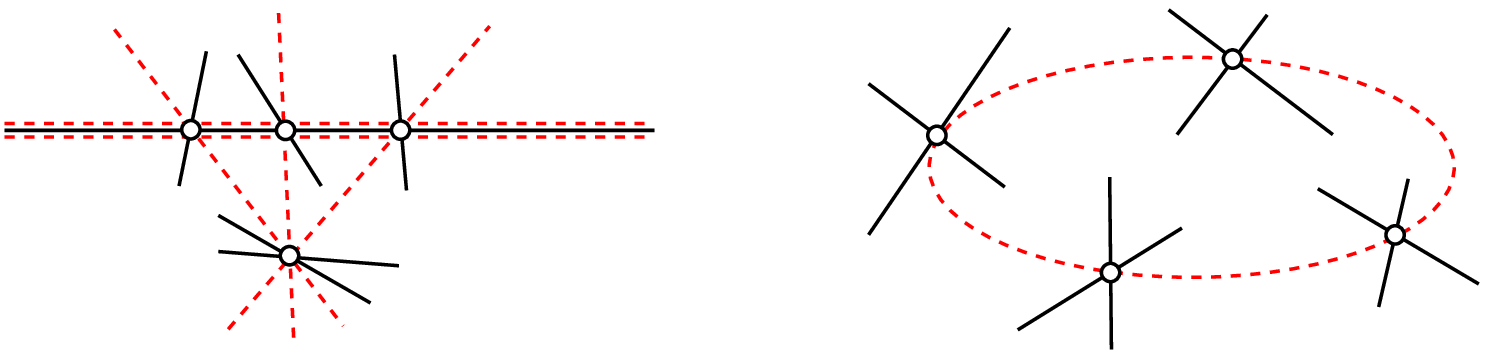}%
\end{picture}%
\setlength{\unitlength}{4144sp}%
\begingroup\makeatletter\ifx\SetFigFont\undefined%
\gdef\SetFigFont#1#2#3#4#5{%
  \reset@font\fontsize{#1}{#2pt}%
  \fontfamily{#3}\fontseries{#4}\fontshape{#5}%
  \selectfont}%
\fi\endgroup%
\begin{picture}(6783,1618)(261,-2089)
\put(6789,-1581){\makebox(0,0)[lb]{\smash{{\SetFigFont{12}{14.4}{\familydefault}{\mddefault}{\updefault}{\color[rgb]{0,0,0}$A_3$}%
}}}}
\put(6715,-915){\makebox(0,0)[lb]{\smash{{\SetFigFont{12}{14.4}{\familydefault}{\mddefault}{\updefault}{\color[rgb]{0,0,0}$E_{1_j}$}%
}}}}
\put(2913,-959){\makebox(0,0)[lb]{\smash{{\SetFigFont{12}{14.4}{\familydefault}{\mddefault}{\updefault}{\color[rgb]{0,0,0}$H_i$}%
}}}}
\put(1606,-953){\makebox(0,0)[lb]{\smash{{\SetFigFont{12}{14.4}{\familydefault}{\mddefault}{\updefault}{\color[rgb]{0,0,0}$A_{i_2}$}%
}}}}
\put(694,-944){\makebox(0,0)[lb]{\smash{{\SetFigFont{12}{14.4}{\familydefault}{\mddefault}{\updefault}{\color[rgb]{0,0,0}$A_{i_1}$}%
}}}}
\put(2164,-1271){\makebox(0,0)[lb]{\smash{{\SetFigFont{12}{14.4}{\familydefault}{\mddefault}{\updefault}{\color[rgb]{0,0,0}$A_{i_3}$}%
}}}}
\put(1168,-1798){\makebox(0,0)[lb]{\smash{{\SetFigFont{12}{14.4}{\familydefault}{\mddefault}{\updefault}{\color[rgb]{0,0,0}$A_{i_4}$}%
}}}}
\put(4240,-1139){\makebox(0,0)[lb]{\smash{{\SetFigFont{12}{14.4}{\familydefault}{\mddefault}{\updefault}{\color[rgb]{0,0,0}$A_1$}%
}}}}
\put(5390,-1902){\makebox(0,0)[lb]{\smash{{\SetFigFont{12}{14.4}{\familydefault}{\mddefault}{\updefault}{\color[rgb]{0,0,0}$A_2$}%
}}}}
\put(6054,-689){\makebox(0,0)[lb]{\smash{{\SetFigFont{12}{14.4}{\familydefault}{\mddefault}{\updefault}{\color[rgb]{0,0,0}$A_4$}%
}}}}
\put(4515,-630){\makebox(0,0)[lb]{\smash{{\SetFigFont{12}{14.4}{\familydefault}{\mddefault}{\updefault}{\color[rgb]{0,0,0}$H_{1_j}$}%
}}}}
\end{picture}%
\end{center}

If there exists a collinear subset $\{A_{i_1},A_{i_2},A_{i_3}\}$ of $A_{4,2}$, then by the generic condition, it must be contained in some line $H_i$. Let $A_{i_4}$ be the remaining point of the set $A_{4,2}$ and let $\mathcal{C}_4$ be the degenerate quintic consisting of the three lines $A_{i_j}A_{i_4}$ ($1\leq j\leq 3$) and of the line $H_i$ with multiplicity $2$. Since $\mathcal{C}_4$ passes through all points in $A_{4,2}$ with multiplicity at least $3$, the inequality \eqref{k,d relation 1} is satisfied. By using \eqref{-last-estimate1}, we get a contradiction.

Now we assume that any subset of $A_{4,2}$ containing three points is not collinear. Let $E_{i_j}=\{\mathsf{e}_{i_j}=0\}$ ($1\leq i\leq 4$, $1\leq j\leq2$) be the eight conics passing through all points of $A_{4,2}$, tangent to the line $H_{i_j}$ at the point $A_i$ ($1\leq i\leq 4$, $1\leq j\leq2$). Let $\mathcal{E}=\{\mathsf{e}=0\}$ be the degenerate curve of degree $16$ consisting of all these $E_{i_j}$. We claim that $f$ does not land in $\mathcal{E}$. Otherwise, it lands in some conic $E_{i_j}$. Since the number of intersection points between $E_{i_j}$ and $\cup_{i=1}^9\, H_i\setminus A_{4,2}$ is $>3$ and since any complement of three distinct points in an irreducible curve is hyperbolic, $f$ must be constant, which is a contradiction.

Letting $z$ be a point in $f^{-1}(A_i)$, we have 
\[
\ord_z(\mathsf{e}_{i_j}\circ f)
\,
\geq
\,
1 \quad\quad {\scriptstyle(1\,\leq\,i\,\leq\, 4\,,\,1\,\leq \, j\,\leq\, 2)}.
\]
By the construction of $E_{i_j}$, if $\ord_z(h_{i_j}\circ f)\geq 2$ for some $1\leq j\leq 2$, then we also have $\ord_z(\mathsf{e}_{i_j}\circ f)\geq 2$. Furthermore, if $\ord_z(h_{i_j}\circ f)\geq 2$ for all $1\leq j\leq 2$, then $\ord_z(\mathsf{e}_{i_j}\circ f)\geq 2$ for all $1\leq i\leq 4$, $1\leq i\leq 2$. Thus, the following inequality holds:
\begin{align*}
\min
\,
\{\ord_z(h_{i_1})\circ f,2\}
+
\min
\,
\{\ord_z(h_{i_2})\circ f,2\}
&
\,
\leq
\,
\dfrac{1}{3}
\sum_{i=1}^4
\sum_{j=1}^2\, \ord_z(\mathsf{e}_{i_j}\circ f)\\
&
\,
=
\,
\dfrac{1}{3}
\ord_z(\mathsf{e}\circ f)\quad\quad
{\scriptstyle(z\,\in\,f^{-1}(A_i))}.
\end{align*}
This implies
\begin{equation*}
\sum_{i\in I}
N_f^{[2]}(r,H_i)
\,
\leq
\,
\dfrac{1}{3}
\,
N_f(r,\mathcal{E}).
\end{equation*}
We proceed similarly as before to derive a contradiction
\begin{align*} 
6
\,
T_f(r)
&
\,
\leq
\,
\sum_{i=1}^{9}
N_f^{[2]}(r,H_i)
+
S_f(r)\\
&
\,
\leq
\,
\dfrac{1}{3}
\,
N_f(r,\mathcal{E})
+
S_f(r)\\
&
\,
\leq
\,
\dfrac{16}{3}
\,
T_f(r)
+
S_f(r).
\end{align*}
Lemma~\ref{conjecture_P2} is thus proved.
\end{proof}
\centering
\bibliographystyle{plain}
\bibliography{article}

\def\cprime{$'$}
\begin{thebibliography}{10}

\bibitem{brody1978}
Robert Brody.
\newblock Compact manifolds and hyperbolicity.
\newblock {\em Trans. Amer. Math. Soc.}, 235:213--219, 1978.

\bibitem{Brody-Green1977}
Robert Brody and Mark Green.
\newblock A family of smooth hyperbolic hypersurfaces in {$P_{3}$}.
\newblock {\em Duke Math. J.}, 44(4):873--874, 1977.

\bibitem{cartan1933}
Henri Cartan.
\newblock Sur les z\'{e}ros des combinaisons lin\'{e}aires de $p$ fonctions
  holomorphes donn\'{e}es.
\newblock {\em Mathematica}, 7:5--31, 1933.

\bibitem{demailly2015proof}
Jean-Pierre Demailly.
\newblock Proof of the {K}obayashi conjecture on the hyperbolicity of very
  general hypersurfaces.
\newblock {\em Preprint arXiv:1501.07625}, 2015.

\bibitem{demailly_goul2000}
Jean-Pierre Demailly and Jawher El~Goul.
\newblock Hyperbolicity of generic surfaces of high degree in projective
  3-space.
\newblock {\em Amer. J. Math.}, 122(3):515--546, 2000.

\bibitem{DMR2010}
Simone Diverio, J\"{o}el Merker, and Erwan Rousseau.
\newblock Effective algebraic degeneracy.
\newblock {\em Inventiones mathematicae}, 180:161--223, 2010.

\bibitem{Diverio-Trapani2010}
Simone Diverio and Stefano Trapani.
\newblock A remark on the codimension of the {G}reen-{G}riffiths locus of
  generic projective hypersurfaces of high degree.
\newblock {\em J. Reine Angew. Math.}, 649:55--61, 2010.

\bibitem{dufresnoy1944}
Jacques Dufresnoy.
\newblock Th\'eorie nouvelle des familles complexes normales. {A}pplications
  \`a l'\'etude des fonctions alg\'ebro\"\i des.
\newblock {\em Ann. Sci. \'Ecole Norm. Sup. (3)}, 61:1--44, 1944.

\bibitem{duval2004}
Julien Duval.
\newblock Une sextique hyperbolique dans {${\rm P}\sp 3(\bold C)$}.
\newblock {\em Math. Ann.}, 330(3):473--476, 2004.

\bibitem{duval2014}
Julien Duval.
\newblock Around {B}rody lemma.
\newblock Preprint 2014.

\bibitem{eremenko2010}
Alexandre Eremenko.
\newblock Brody curves omitting hyperplanes.
\newblock {\em Ann. Acad. Sci. Fenn. Math.}, 35(2):565--570, 2010.

\bibitem{Fujimoto1972}
Hirotaka Fujimoto.
\newblock On holomorphic maps into a taut complex space.
\newblock {\em Nagoya Math. J.}, 46:49--61, 1972.

\bibitem{fujimoto2001}
Hirotaka Fujimoto.
\newblock A family of hyperbolic hypersurfaces in the complex projective space.
\newblock {\em Complex Variables Theory Appl.}, 43(3-4):273--283, 2001.
\newblock The Chuang special issue.

\bibitem{Green1972}
Mark~Lee Green.
\newblock Holomorphic maps into complex projective space omitting hyperplanes.
\newblock {\em Trans. Amer. Math. Soc.}, 169:89--103, 1972.

\bibitem{kobayashi1998}
Shoshichi Kobayashi.
\newblock {\em Hyperbolic complex spaces}, volume 318 of {\em Grundlehren der
  Mathematischen Wissenschaften [Fundamental Principles of Mathematical
  Sciences]}.
\newblock Springer-Verlag, Berlin, 1998.

\bibitem{masuda_noguchi1996}
Kazuo Masuda and Junjiro Noguchi.
\newblock A construction of hyperbolic hypersurface of {${\bf P}\sp n({\bf
  C})$}.
\newblock {\em Math. Ann.}, 304(2):339--362, 1996.

\bibitem{mihaipaun2008}
Mihai P{\u{a}}un.
\newblock Vector fields on the total space of hypersurfaces in the projective
  space and hyperbolicity.
\newblock {\em Math. Ann.}, 340(4):875--892, 2008.

\bibitem{Rousseau2007}
Erwan Rousseau.
\newblock Weak analytic hyperbolicity of generic hypersurfaces of high degree
  in {$\Bbb P^4$}.
\newblock {\em Ann. Fac. Sci. Toulouse Math. (6)}, 16(2):369--383, 2007.

\bibitem{shiffman_zaidenberg2002_pn}
Bernard Shiffman and Mikhail Zaidenberg.
\newblock Hyperbolic hypersurfaces in {$\Bbb P\sp n$} of {F}ermat-{W}aring
  type.
\newblock {\em Proc. Amer. Math. Soc.}, 130(7):2031--2035 (electronic), 2002.

\bibitem{shiffman_zaidenberg2005}
Bernard Shiffman and Mikhail Zaidenberg.
\newblock New examples of {K}obayashi hyperbolic surfaces in {$\Bbb C\rm P\sp
  3$}.
\newblock {\em Funct. Anal. Appl.}, 39(1):90--94, 2005.

\bibitem{siu2015}
Yum-Tong Siu.
\newblock Hyperbolicity of generic high-degree hypersurfaces in complex
  projective space.
\newblock {\em Inventiones mathematicae}, pages 1--98, 2015.

\bibitem{siu_yeung1997}
Yum-Tong Siu and Sai-Kee Yeung.
\newblock Defects for ample divisors of abelian varieties, {S}chwarz lemma, and
  hyperbolic hypersurfaces of low degrees.
\newblock {\em Amer. J. Math.}, 119(5):1139--1172, 1997.

\bibitem{zaidenberg2003hyperbolic}
Mikhail Zaidenberg.
\newblock Hyperbolic surfaces in $\mathbb{P}^3$: examples.
\newblock {\em Preprint arXiv math.AG/0311394}, 2003.

\end{thebibliography}
\Addresses
\end{document}